\RequirePackage{rotating}
\documentclass[11pt]{amsart}
\usepackage[colorlinks, linkcolor=blue, citecolor=blue, urlcolor=red!80!gray, pagebackref]{hyperref}
\usepackage{amssymb}
\usepackage{amsmath}
\usepackage{mathtools}
\usepackage{tabularray}
\usepackage[initials]{amsrefs}
\usepackage{ytableau}
\usepackage{subcaption}
\usepackage{stmaryrd}
\usepackage{tikz}
\usetikzlibrary{positioning,fit}
\usepackage[indent]{parskip}
\usepackage{url}
\usepackage{enumitem}
\usepackage{fullpage}
\usepackage{rotating}
\usepackage{float}

\theoremstyle{plain}
\newtheorem{theorem}{Theorem}[section]
\newtheorem{prop}[theorem]{Proposition}
\newtheorem{lemma}[theorem]{Lemma}

\theoremstyle{definition}
\newtheorem{dfn}[theorem]{Definition}

\newtheorem{rem}[theorem]{Remark}
\numberwithin{equation}{section}

\newcommand{\C}{K}
\DeclareMathOperator{\GL}{GL}
\DeclareMathOperator{\SL}{SL}
\DeclareMathOperator{\Sp}{Sp}
\DeclareMathOperator{\RSK}{RSK}
\renewcommand{\O}{\operatorname{O}}
\DeclareMathOperator{\SO}{SO}
\renewcommand{\phi}{\varphi}
\DeclareMathOperator{\SSYT}{SSYT}
\DeclareMathOperator{\SYT}{SYT}
\newcommand{\posarrow}[2]{ \underset{\substack{\textstyle\uparrow\\\hidewidth\mathstrut#2\hidewidth}}{#1}}

\begin{document}

\title{Tensor invariants for classical groups revisited}

\author{William Q.~Erickson}
\address{
William Q.~Erickson\\
Department of Mathematics\\
Baylor University \\ 
One Bear Place \#97328\\
Waco, TX 76798} 
\email{Will\_Erickson@baylor.edu}

\author{Markus Hunziker}
\address{
Markus Hunziker\\
Department of Mathematics\\
Baylor University \\ 
One Bear Place \#97328\\
Waco, TX 76798} 
\email{Markus\_Hunziker@baylor.edu}

\begin{abstract}
We reconsider an old problem, namely the dimension of the $G$-invariant subspace in $V^{\otimes p} \otimes V^{*\otimes q}$, where $G$ is one of the classical groups $\GL(V)$, $\SL(V)$, $\O(V)$, $\SO(V)$, or $\Sp(V)$.
Spanning sets for the invariant subspace have long been well known, but linear bases are more delicate.
The main contribution of this paper is a combinatorial realization of linear bases via standard Young tableaux and arc diagrams, in a uniform manner for all five classical groups.
As a secondary contribution, we survey the many equivalent ways --- some old, some new --- to enumerate the elements in these bases.
\end{abstract}



\maketitle

\tableofcontents

\section{Introduction}

The aim of this paper is to write down an exhaustive list of combinatorial answers to the following classical problem: if we are given a finite-dimensional vector space $V$, then for each of the classical groups $G \subseteq \GL(V)$, what is the dimension of the invariant subspace $(V^{\otimes p} \otimes V^{*\otimes q})^G$?
It is straightforward to specify a spanning set in each case; this is the content of Weyl's first fundamental theorems for tensor invariants~\cite{Weyl}.
Due to the relations, however, which are described in the second fundamental theorems, it is less straightforward to determine a linear basis.
Surprisingly, we are unable to find a reference anywhere in the invariant theory literature which solves this combinatorial problem in a uniform manner for all five classical groups.
As a main result of the present paper (Theorems~\ref{thm:multilinear GL O Sp} and~\ref{thm:multilinear SL SO}), we give combinatorial realizations of linear bases as sets of arc diagrams, uniformly across all classical groups $\GL(V)$, $\SL(V)$, $\O(V)$, $\SO(V)$, and $\Sp(V)$.
As another main result (Theorems~\ref{thm:enumerate GL O Sp} and \ref{thm:enumerate SL SO}), we catalog several equivalent ways to enumerate these linear bases.

Our paper takes the following approach.
Let $V$ be a finite-dimensional vector space over a field $\C$ of characteristic 0, and let $W \coloneqq V^{*p} \oplus V^{q}$.
We begin by addressing the (more general) problem of \emph{polynomial invariants}, that is, elements of the invariant ring $\C[W]^G$.
For each classical group $G$, there is a canonical linear basis for $\C[W]^G$ consisting of \emph{standard monomials}, due to DeConcini--Procesi~\cite{DeConciniProcesi}, obtained via the straightening laws of Doubilet--Rota--Stein~\cite{DRS}.
The technique came to be known as \emph{standard monomial theory}, introduced independently by Seshadri~\cite{Seshadri}.
A highly detailed treatment of the standard monomial bases for $\C[W]^G$ is also given by Lakshmibai--Raghavan~\cite{Lakshmibai}*{Ch.~10--12}.
We devote Section~\ref{sec:CIT standard monomials} to recalling these bases of standard monomials, and to realizing them as sets of semistandard tableaux which we denote by $\mathcal{S}^G$.

In Section~\ref{sec:RSK ordinary monomials}, we establish another linear basis for $\C[W]^G$, which we denote by $\mathcal{B}^G$.
Unlike the basis $\mathcal{S}^G$, the basis $\mathcal{B}^G$ consists of \emph{ordinary monomials}, meaning true monomials in the contractions and determinants that generate $\C[W]^G$.
(By contrast, the standard monomials in $\mathcal{S}^G$ are not actually monomials in these generators.)
We obtain $\mathcal{B}^G$ by applying certain variants of the Robinson--Schensted--Knuth (RSK) correspondence~\cites{Knuth} to the basis $\mathcal{S}^G$: in particular, each RSK-type correspondence yields a multidegree-preserving bijection between the tableaux in $\mathcal{S}^G$ and the degree matrices (in a set we call $\mathcal{M}^G$) of the ordinary monomials in $\mathcal{B}^G$:
\begin{equation}
\label{proof picture}
\begin{array}{ccccc}
    \text{tableaux in $\mathcal{S}^G$} & \xleftrightarrow{{\rm RSK}} & \text{$\mathbb{N}$-matrices in $\mathcal{M}^G$} & & \\
   \updownarrow & & \updownarrow & & \\
   \text{standard monomials} & & \text{ordinary monomials in $\mathcal{B}^G$} & \longleftrightarrow & \text{arc diagrams}
\end{array}
\end{equation}

\noindent This view of RSK as a transformation between linear bases seems to have been introduced by Sturmfels~\cite{Sturmfels}, Herzog--Trung~\cite{HerzogTrung}, and Conca~\cite{Conca} in the context of determinantal rings; to the best of our knowledge, however, the basis $\mathcal{B}^G$ has not arisen in the invariant theory literature.

From our viewpoint, the combinatorial advantage of $\mathcal{B}^G$ is that its elements can naturally be viewed as \emph{arc diagrams}, as indicated in~\eqref{proof picture}.
This is due to the fact that the basis elements are ordinary monomials in certain contractions $f_{ij}$ (each of which can be viewed as an arc between vertices $i$ and $j$) and in certain determinants (which can be viewed as hyperedges).
The theorems in Section~\ref{sec:polynomial bases} realize $\mathcal{B}^G$ as the set of these arc diagrams.
In order to specialize from polynomial invariants to tensor invariants --- the motivating problem in this paper --- we simply restrict our attention to arc diagrams that are 1-regular (i.e., those in which each vertex has degree 1).
In this way, our theorems in Section~\ref{sec:tensor bases} state how to realize basis elements for tensor invariants as 1-regular arc diagrams.
As an example, where $\dim V = 3$ and $G = \SO(V)$ acts on $V^{\otimes 15}$, a typical $G$-invariant basis element is shown below, displayed beneath its corresponding tableau and arc diagram:

\vspace{-2ex}

\begin{center}
\begin{tikzpicture}[baseline, scale = .75, every node/.append style={transform shape}]
    \node at (0,0) {$\ytableausetup{nosmalltableaux} 
\begin{ytableau} 
 *(lightgray)1&2&5&10&11&12&13\\
 *(lightgray)3&4&6&14&15\\ 
 *(lightgray)7&8&9 
 \end{ytableau}$};
\end{tikzpicture}
\quad
\begin{tikzpicture}[scale=.6,-, node distance = .8 cm,
  thick, inner sep=0 pt, plain node/.style={circle,draw,font=\sffamily\bfseries,fill=white,minimum size = 0.6cm}, painted node/.style={minimum size=0.6cm,inner sep=0pt,circle,draw,font=\sffamily\bfseries,fill=lightgray, text=black},every node/.append style={transform shape},bend left = 60, baseline]

\node[painted node] (1) {1};
\node[plain node] (2) [right of=1] {2};
\node[painted node] (3) [right of=2] {3};
\node[plain node] (4) [right of=3] {4};
\node[plain node] (5) [right of=4] {5};
\node[plain node] (6) [right of=5] {6};
\node[painted node] (7) [right of=6] {7};
\node[plain node] (8) [right of=7] {8};
\node[plain node] (9) [right of=8] {9};
\node[plain node] (10) [right of=9] {10};
\node[plain node] (11) [right of=10] {11};
\node[plain node] (12) [right of=11] {12};
\node[plain node] (13) [right of=12] {13};
\node[plain node] (14) [right of=13] {14};
\node[plain node] (15) [right of=14] {15};

\draw (2) to (5);
\draw (4) to (13);
\draw (6) to (12);
\draw (8) to (15);
\draw (9) to (14);
\draw (10) to (11);
\end{tikzpicture}
\end{center}

\vspace{-2ex}
\[
v_1 \otimes \cdots \otimes v_{15} \longmapsto \det(v_1, v_3, v_7) \: b(v_2, v_5) \: b(v_4, v_{13}) \: b(v_6, v_{12}) \: b(v_8, v_{15}) \: b(v_9, v_{14}) \: b(v_{10},v_{11})
\]

\noindent The (shaded) first column of the tableau corresponds to the (shaded) hyperedge $\{1,3,7\}$ in the arc diagram, which in turn corresponds to the determinant in the explicit basis element written beneath.
(Here we employ the natural $G$-module isomorphism $V^{\otimes 15} \cong (V^{\otimes 15})^*$ to write a tensor invariant as a product of contractions, where $b( \:, \:)$ denotes the nondegenerate symmetric bilinear form preserved by $\SO(V)$.)
Note that these contractions can be read off directly from the arcs in the diagram.
The RSK-type correspondence takes the unshaded part of the tableau to the adjacency matrix of the arc diagram. 

We turn in Section~\ref{sec:enumeration} to the question of enumeration: in particular, we collect alternative interpretations of the dimension of the tensor invariants from the combinatorics literature, which we present together in Theorems~\ref{thm:enumerate GL O Sp} and~\ref{thm:enumerate SL SO}.
In the appendix, we include several tables where we organize and illustrate the results in these enumerative theorems.
In Table~\ref{table:OEIS}, we collect our dimension formulas into ``sequences of sequences'' which can be found in the Online Encyclopedia of Integer Sequences (OEIS).
For certain parameters (when $\dim V$ is quite small), the OEIS entries include a wealth of combinatorial interpretations; for others, however, the OEIS entry (if it exists) was previously created solely as the result of a computation in the program LiE.
Tables~\ref{table:SL example}--\ref{table:Sp example} are concrete examples for given parameters, showing the arc diagrams and corresponding tableaux which constitute our linear bases.

As mentioned above, it seems that our bases $\mathcal{B}^G$ and their corresponding arc diagrams are new;
regardless, our larger purpose here is to shed some light on a classical problem, and in doing so, to unify related enumerative results that are currently scattered in the vast literature.
This paper also offers an additional perspective on recent work~\cites{BostanEtAl,BDM24,LinshawSong,RSS12,Schilling23} pertaining to canonical bases, standard monomials, and invariant theory.
These works build upon a long and rich legacy in combinatorial invariant theory, from the 1970s (see De Concini--Procesi~\cite{DeConciniProcesi}, D\'{e}sarm\'{e}nien--Kung--Rota~\cite{DKR}, and Stanley~\cite{StanleyCombInvThy}) to more recent graphical approaches such as Kuperberg's webs~\cite{Kuperberg} and the wave diagrams defined in~\cite{MihailovsThesis}*{Ch.~2}.
See also the papers~\cites{WestburyGL,RubeyWestburySymplectic} studying combinatorial properties of tensor invariants for $\GL(V)$ and $\Sp(V)$.
On the enumerative side, we refer the reader to the survey~\cite{Post} of arc diagrams, Young tableaux, and lattice walks.
In some form, the use of graphs to encode polynomial invariants dates back at least to Sylvester~\cites{Sylvester,GY} and Kempe~\cite{Kempe};
likewise, some of the fundamental theorems proved by Weyl were known much earlier, for instance to Cayley~\cite{Cayley}, Clebsch~\cite{Clebsch}, Gordan~\cite{Gordan}, and Study~\cite{Study}; in particular, we refer the reader to the elegantly illustrated historical overview~\cite{AA}*{pp.~13--15} of the group $\SL_2$.

Throughout the paper, we present our major propositions/theorems in pairs, where the first result pertains to $\GL(V)$, $\O(V)$, and $\Sp(V)$, and the second to $\SL(V)$ and $\SO(V)$.
In each result, part (a) concerns $\GL(V)$ or $\SL(V)$, part (b) concerns $\O(V)$ or $\SO(V)$, and part (c) concerns $\Sp(V)$.

\subsection*{Acknowledgments}

The authors are deeply grateful to the anonymous referees for their suggestions which vastly improved the structure and clarity of the paper.

\section{Classical invariant theory and standard monomials}
\label{sec:CIT standard monomials}

\subsection*{The classical groups}

Throughout the paper, we let $\GL(V)$ denote the \emph{general linear group} of invertible linear operators on a finite-dimensional vector space $V$ over a field $\C$ of characteristic 0.
If $V$ is equipped with a nondegenerate symmetric bilinear form $b$, then we write $\O(V) \subset \GL(V)$ to denote the \emph{orthogonal group} preserving $b$.
If the form is instead a nondegenerate skew-symmetric bilinear form $\omega$, then we write $\Sp(V) \subset \GL(V)$ to denote the \emph{symplectic group} preserving $\omega$.
When $G = \O(V)$ or $\Sp(V)$, this bilinear form induces a canonical $G$-module isomorphism $V \cong V^*$, and for this reason we will not distinguish between $V$ and its dual space $V^*$ when considering these groups.
The \emph{special linear group} $\SL(V) \subset \GL(V)$ and \emph{special orthogonal group} $\SO(V) \subset \O(V)$ are the subgroups consisting of operators with determinant 1.
Collectively, the five families described above are known as the \emph{classical groups}.
Throughout the paper, we use the parameter $n$ to denote the dimension of $V$, except in the case $G = \Sp(V)$, where we put  $\dim V = 2n$.

\subsection*{Polynomial invariants and multidegree}

Let $\mathbb{N}$ denote the set of nonnegative integers.
We write $V^m$ or $V^{*m}$ to denote the direct sum of $m$ copies of $V$ or its dual space $V^*$, respectively.
In the following exposition, when giving a unified treatment of the classical groups, we set
\[
W \coloneqq \begin{cases}
    V^{*p} \oplus V^q, & G = \GL(V) \text{ or } \SL(V),\\
    V^m, & G = \O(V) \text{ or } \SO(V) \text{ or } \Sp(V).
\end{cases}
\]
In the first case above, we write an element of $W$ as $(\phi_1, \ldots, \phi_p, v_1, \ldots, v_q)$, and in the second case we write an element of $W$ as $(v_1, \ldots, v_m)$.
The natural diagonal action of $G$ on $W$ extends to a linear action on the space $\C[W]$ of $\C$-valued polynomial functions on $W$, via $(g \cdot f)(w) = f(g^{-1} \cdot w)$ for all $g \in G$, $f \in \C[W]$, and $w \in W$.
The object of interest in Sections~\ref{sec:CIT standard monomials}--\ref{sec:polynomial bases} is the ring of invariants
\[
\C[W]^G \coloneqq \{f \in \C[W] : f(g \cdot w) = f(w) \text{ for all $g \in G$ and $w \in W$}\}.
\]

There is a natural multigradation on $\C[W]^G$, which we will generally designate with a boldface $\mathbf{d}$.
For $G = \GL(V)$ or $\SL(V)$, the action of the torus $(\C^\times)^p \times (\C^\times)^q$ leads to an $\mathbb{N}^p \times \mathbb{N}^q$-gradation on $\C[V^{*p} \oplus V^q]$; see~\cite{GW}*{p.~256}.
For $G = \GL(V)$ and $\SL(V)$, then, the symbol $\mathbf{d}$ actually represents an ordered pair $\mathbf{d} = (\mathbf{d}_p, \mathbf{d}_q) \in \mathbb{N}^p \times \mathbb{N}^q$, and we write $\C[V^{*p} \oplus V^q]_{(\mathbf{d}_p, \mathbf{d}_q)}$ to denote the corresponding homogeneous multigraded component.
When $G = \O(V)$ or $\SO(V)$ or $\Sp(V)$, we write $\C[V^m]_\mathbf{d}$ where $\mathbf{d} \in \mathbb{N}^m$.
In general, we write $\C[W]^G_{\mathbf{d}} \coloneqq \C[W]^G \cap \C[W]_{\mathbf{d}}$.

\subsection*{Fundamental theorems and Gr\"obner bases}

For $G = \GL(V)$, $\O(V)$, or $\Sp(V)$, we describe the structure of $\C[W]^G$ in Proposition~\ref{prop:FFT SFT GL O Sp} below.
In each case, $\C[W]^G$ is generated by certain contractions $f_{ij}$, and the relations are given by the vanishing of certain determinants (or Pfaffians) of minors in the $f_{ij}$'s.
The generation by the $f_{ij}$'s is the content of the first fundamental theorem (FFT) of classical invariant theory, and the relations are the content of the second fundamental theorem (SFT).
In the case where $\C = \mathbb{C}$, the FFT and SFT are given in Weyl's book~\cite{Weyl}; see also~\cite{Kraft}.
In the proposition below, we use the term \emph{$(n+1)$-minor} to mean the determinant of an arbitrary $(n+1) \times (n+1)$ minor of a matrix.
By a \emph{$2(n+1)$-Pfaffian} of an alternating matrix, we mean the Pfaffian of a $2(n+1) \times 2(n+1)$ alternating matrix obtained by choosing $2(n+1)$ many indices as its rows and columns.
We use $x_{ij}$ as indeterminates.

\begin{prop}
\label{prop:FFT SFT GL O Sp}

Let $G = \GL(V)$, $\O(V)$, or $\Sp(V)$.
Let $\dim V = n$ if $G = \GL(V)$ or $\O(V)$, and let $\dim V = 2n$ if $G = \Sp(V)$.
We have an isomorphism of algebras
\begin{align}
\label{quotient}
\begin{split}
       \C[\{x_{ij}\}] / \langle \mathcal{R} \rangle & \longrightarrow \C[W]^G, \\ 
   x_{ij} & \longmapsto f_{ij},
\end{split}
\end{align}
where the generators $f_{ij}$ and relations $\mathcal{R}$ are given in Table~\ref{table:FT}.
Moreover, with respect to the lexicographic order induced by the variable order given in the table, $\mathcal{R}$ is a Gr\"obner basis for $\langle \mathcal{R} \rangle$.

\end{prop}

\begin{table}[t]
    \begin{center}
\resizebox{\linewidth}{!}{
\begin{tblr}{colspec={|Q[m,c]|Q[m,c]|Q[m,c]|Q[m,c]|},stretch=1.5}

\hline

$G$ & $\GL(V)$ & $\O(V)$ & $\Sp(V)$ \\ \hline[2pt]

$(i,j)$ & $1 \leq i \leq p, \: 1 \leq j \leq q$ & $1 \leq i \leq j \leq m$ & $1 \leq i < j \leq m$ \\ \hline

$f_{ij}$ & $(\phi_1, \ldots, \phi_p, v_1, \ldots, v_q) \mapsto \phi_i(v_j)$ & $(v_1, \ldots, v_m) \mapsto b(v_i, v_j)$ & $(v_1, \ldots, v_m) \mapsto \omega(v_i, v_j)$ \\ \hline

$\mathcal{R}$ & $\left\{ \parbox{3.1cm}{\centering \normalfont $(n+1)$-minors \\ of the matrix $[x_{ij}]$}\right\}$
&
$\left\{ \parbox{3cm}{\centering \normalfont $(n+1)$-minors \\ of the symmetric matrix $[x_{ij}]$}\right\}$
&
$\left\{ \parbox{3.2cm}{\centering \normalfont $2(n+1)$-Pfaffians \\ of the alternating matrix $[x_{ij}]$}\right\}$ \\ \hline

{\normalfont Var.\\order} & $\begin{aligned}
    &x_{1q} > \cdots > x_{11} > \\
    & x_{2q} >  \cdots > x_{21} > \\
    &\phantom{x_{pq} >} \cdots > x_{p1}
\end{aligned}$ & $\begin{aligned}
    &x_{11} > \cdots > x_{1m} > \\
    & x_{22} >  \cdots > x_{2m} >\\
    & \phantom{x_{22} >}  \cdots > x_{mm} \\
\end{aligned}$ & $\begin{aligned}
    &x_{1m} > \cdots > x_{12} > \\
    & x_{2m} >  \cdots > x_{23} >\\
    & \phantom{x_{2m}. >}  \cdots > x_{m-1,m} \\
\end{aligned}$ \\ \hline

\end{tblr}
}
\end{center}
    \caption{Data accompanying Proposition~\ref{prop:FFT SFT GL O Sp}.}
    \label{table:FT}
\end{table}

\begin{rem} 
\label{rem: [fij]}
For $G = \O(V)$ (resp., $\Sp(V)$), note that $[x_{ij}]$ is the $m \times m$ matrix of indeterminates whose lower triangular entries are obtained by putting $x_{ji} = x_{ij}$ (resp., $x_{ji} = -x_{ij}$).
\end{rem}

\begin{proof}
    
    The generators $f_{ij}$ are given by the FFT; see DeConcini--Procesi~\cite{DeConciniProcesi}, in particular Theorems 3.1, 5.6(i), and 6.6.
    The minimal set $\mathcal{R}$ of relations is given by the SFT; again, see~\cite{DeConciniProcesi}, in particular Theorems 3.4, 5.7, and 6.7, respectively.
    The isomorphism~\eqref{quotient} follows immediately.

    The fact that $\mathcal{R}$ is a Gr\"obner basis follows from well-known results concerning determinantal rings.
    The rings on the left-hand side of~\eqref{quotient}, namely quotients of polynomial rings by ideals $\langle \mathcal{R} \rangle$ generated by minors or Pfaffians, were studied intensely in the 1990s, although not in the context of classical invariant theory.
    In particular, for the quotient corresponding to each of the three groups $G$ in Table~\ref{table:FT}, it is known that $\mathcal{R}$ is a Gr\"obner basis for $\langle \mathcal{R} \rangle$, with respect to the monomial ordering induced by the variable order in the table above.
    For the quotient corresponding to $G = \GL(V)$, see Sturmfels~\cite{Sturmfels}*{Thm.~1}; for $G = \O(V)$, see Conca~\cite{Conca}*{Thm.~2.8}; for $G = \Sp(V)$, see Herzog--Trung~\cite{HerzogTrung}*{Thm.~5.1}.
\end{proof}

\subsection*{Standard monomials via tableaux}

A \emph{partition} is a weakly decreasing sequence of positive integers, which we denote by $\lambda = (\lambda_1, \ldots, \lambda_r)$.
We call $r$ the \emph{length} of $\lambda$, written as $\ell(\lambda) = r$.
We use the shorthand $(a^r) \coloneqq (a, \ldots, a)$, and we can take a componentwise sum $\lambda + \mu$ by setting $(\lambda + \mu)_i = \lambda_i + \mu_i$, padding the shorter partition with zeros as necessary.
If $s = \sum_i \lambda_i$, then we say that $\lambda$ is a partition of $s$, written as $\lambda \vdash s$.
By a \emph{semistandard Young tableau} of shape $\lambda$, we mean a left-justified arrangement of boxes, where the $i$th row from the top contains $\lambda_i$ many boxes, and where the boxes are filled with entries such that every row is weakly increasing and every column is strictly increasing.
Let $\SSYT(\lambda, m)$ denote the set of all semistandard Young tableaux of shape $\lambda$, with entries taken from the set $[m] \coloneqq \{1, \ldots, m\}$.
The \emph{weight} of a tableau $T \in \SSYT(\lambda, m)$ is the vector ${\rm wt}(T) \in \mathbb{N}^m$ whose $i$th component is the number of occurrences of the entry $i$ in $T$.
In the case $G = \GL(V)$, we will study ordered pairs $(T,U) \in \SSYT(\lambda, p) \times \SSYT(\lambda,q)$, typically called \emph{bitableaux} in the literature; the weight of a bitableau is the ordered pair of weights ${\rm wt}(T,U) = ({\rm wt}(T), {\rm wt}(U))$.
When we refer to the \emph{rows} and \emph{columns} of a partition $\lambda$, we mean the rows and columns of a tableau with shape $\lambda$.
In particular, $\lambda$ is said to have \emph{even rows} (resp., columns) if every row (resp., column) of $\lambda$ contains an even number of boxes.

There is a natural way to view a semistandard (bi)tableau as an element of $\C[W]^G$, by viewing each tableau column as the set of row/column indices determining a minor (or Pfaffian) of the matrix formed by the generators $f_{ij}$, and then taking the product over all the columns in the (bi)tableau.
The elements of $\C[W]^G$ obtained in this way are called \emph{standard monomials}, which explains our notation $\mathcal{S}^G$ in Proposition~\ref{prop:standard monomial basis} below.
In particular, if $T$ is a semistandard tableau of shape $\lambda$, then let $T_\ell$ denote the set of entries in the $\ell$th column of $T$.
We sometimes write a tableau as the concatenation of its columns from left to right: $T = T_1 \cdots T_{\lambda_1}$.
If $I$ and $J$ are sets of positive integers such that $|I| = |J| \leq n$, then define the minor
    \begin{equation}
    \label{minor def}
        {\rm m}(I,J) \coloneqq \det[f_{ij}]_{\substack{i \in I\\ j \in J}}.
    \end{equation}
Likewise, for $|I|$ an even number at most $2n$, define the Pfaffian
\begin{equation}
    \label{pffafian def}
    {\rm pf}(I) \coloneqq {\rm pf}[f_{ij}]_{i,j \in I}.
\end{equation}
Note that for $G = \O(V)$ or $\Sp(V)$, the matrix $[f_{ij}]$ is obtained by extending the definition of $f_{ij}$ to all $1 \leq i,j \leq m$, and is therefore symmetric or alternating, respectively.
This follows from the symmetry (resp., skew-symmetry) of the bilinear form $b$ (resp., $\omega$).

\begin{prop}
    \label{prop:standard monomial basis}
    Let $G = \GL(V)$, $\O(V)$, or $\Sp(V)$, with $n$ as in Proposition~\ref{prop:FFT SFT GL O Sp}.
    A linear basis for $\C[W]^G$ is given by the set $\mathcal{S}^G$ of (bi)tableaux, viewed as the set of \emph{standard monomials} in $\C[W]^G$ via the following identification:
    
    \begin{center} 
\resizebox{\linewidth}{!}{
\begin{tblr}{colspec={|Q[m,c]|Q[m,c]|Q[m,c]|Q[m,c]|},stretch=1.5}

\hline

$G$ & $\GL(V)$ & $\O(V)$ & $\Sp(V)$ \\ \hline[2pt]

$\mathcal{S}^{G}$ & \quad $\displaystyle \bigcup_{\mathclap{\substack{\lambda : \\ \ell(\lambda) \leq n}}} \SSYT(\lambda,p) \times \SSYT(\lambda,q)$ & $\displaystyle \bigcup_{\mathclap{\substack{\lambda: \\ \ell(\lambda) \leq n, \\ \textup{even rows}}}} \SSYT(\lambda, m)$ & \qquad $\displaystyle \bigcup_{\mathclap{\substack{\lambda: \\ \ell(\lambda) \leq 2n, \\ \textup{even columns}}}} \SSYT(\lambda, m)$ \\ \hline

{\normalfont As elements \\ of $\C[W]^G$} &  
$\displaystyle (T,U) = \prod_{\ell=1}^{\lambda_1} {\rm m}(T_\ell,U_\ell)$
& 
$\displaystyle T = \prod_{\ell=1}^{\lambda_1/2} {\rm m}(T_{2\ell-1}, T_{2\ell})$
& $T = \displaystyle \prod_{\ell=1}^{\lambda_1} {\rm pf}(T_\ell)$ \\ \hline

\end{tblr}
}
\end{center}

    \noindent Moreover, a linear basis for $\C[W]^G_{\mathbf{d}}$ is given by the subset
    \[
        \mathcal{S}^{G}_{\mathbf{d}} \coloneqq \{ \textup{(bi)tableaux in $\mathcal{S}^{G}$ with weight $\mathbf{d}$}\}.
    \]
    
    \end{prop}
    
\begin{proof}
    The fact that $\mathcal{S}^{G}$ furnishes a linear basis for $\C[W]^G$ is well known in the literature; see DeConcini--Procesi~\cite{DeConciniProcesi}, in particular Theorems 1.2, 5.1, and 6.5.
    (See also~\cite{Procesi}*{Ch.~13} for an updated exposition.)
    For even greater detail, see Lakshmibai--Raghavan~\cite{Lakshmibai}, in particular Theorems 10.3.1.4, 10.4.0.4, and 10.5.0.2(iii).

    To prove that $\mathcal{S}^{G}_{\mathbf{d}}$ furnishes a basis for the multigraded component $\C[W]^G_{\mathbf{d}}$, we must show that the weight of a (bi)tableau in $\mathcal{S}^{G}$ equals the multidegree of its corresponding basis element in $\C[W]^G$.
    First let $G = \GL(V)$.
    Each generator $f_{ij}$ has multidegree $(\mathbf{e}_i, \mathbf{e}_j)$, where $\mathbf{e}_i$ denotes the tuple with 1 in the $i$th component and 0's elsewhere.
    Let $(T,U) \in \mathcal{S}^{\GL(V)}$.
    For each $\ell$, the minor ${\rm m}(T_\ell, U_\ell)$ consists entirely of terms with multidegree $(\sum_{i \in T_\ell} \mathbf{e}_i, \sum_{j \in U_\ell} \mathbf{e}_j)$, and is therefore homogeneous.
    Moreover, its multidegree is obtained by counting the occurrences of each entry in $T_\ell$ and $U_\ell$.
    Therefore, by taking the product of these minors over all columns $\ell = 1, \ldots, \lambda_1$ and adding their multidegrees, we obtain a homogeneous polynomial whose multidegree is $({\rm wt}(T), {\rm wt}(U)) \eqcolon {\rm wt}(T,U)$, which completes the proof.
    The argument for $\O(V)$ and $\Sp(V)$ is identical, upon observing that each generator $f_{ij}$ has multidegree $\mathbf{e}_i + \mathbf{e}_j$.
\end{proof}

\begin{rem}
    Experts will notice that in defining $\mathcal{S}^{G}$, we have transposed the tableaux as they are often presented in the literature of invariant theory or determinantal rings.
    In~\cites{Procesi,DeConciniProcesi,Sturmfels,HerzogTrung,Conca}, for example, the tableaux are row-strict rather than column-strict, so that the determinants (or Pfaffians) correspond to rows rather than to columns.
    We choose to use column-strict tableaux throughout this paper, in order to follow the typical sense of the term ``semistandard'' in combinatorics.
    The tableaux appearing in $\mathcal{S}^{\O(V)}$, all of which have even rows, are (upon transposing to obtain even \emph{columns}) also known as ``d-tableaux,'' where the ``d'' stands for ``double''~\cite{Conca}*{Def.~1.5},
    or as ``dual tableaux''~\cites{Burge,Knuth}.
\end{rem}

In order to extend Proposition~\ref{prop:standard monomial basis} to the groups $\SL(V)$ and $\SO(V)$, we will allow the standard monomials in $\mathcal{S}^{\GL(V)}$ and $\mathcal{S}^{\O(V)}$ to be multiplied by certain determinants in the covectors $\phi_i$ or the vectors $v_j$.
We use the symbols $A$ and $B$ (with or without subscripts) to denote a set of $n$ positive integers (equivalently, a tableau column of length $n$).
We recall the \emph{tableau order} defined on finite sets of positive integers, viewed as tableau columns with entries increasing from top to bottom:
\begin{equation}
    \label{tableau order}
    A \leq A' \text{ if and only if $A$ can appear to the left of $A'$ in a semistandard tableau}.
\end{equation}
Given a semistandard tableau $T$, and assuming that $A_1 \leq \cdots \leq A_r \leq T_1$, we write $A_1 \cdots A_r T$ to denote the semistandard tableau obtained from $T$ by prepending the columns $A_1, \ldots, A_r$ from left to right.

\begin{prop}
    \label{prop:standard monomials SL SO} Let $n = \dim V$.

    \begin{enumerate}[label=\textup{(\alph*)}]
        \item \label{subprop:standard monomials SL SO a} Let $G = \SL(V)$.
        Let the contractions $f_{ij}$ be as defined in Table~\ref{table:FT} for $\GL(V)$, and let the minors ${\rm m}(I,J)$ be as defined in~\eqref{minor def}.
        For $A = \{a_1, \ldots, a_n\} \subseteq [p]$ and $B = \{b_1, \ldots, b_n\} \subseteq [q]$, define the functions
        \begin{align}
        \label{det SL}
        \begin{split}
            {\det}^*(A) &: (\phi_1, \ldots, \phi_p, v_1, \ldots, v_q) \longmapsto \det( \phi_{a_1}, \ldots, \phi_{a_n}),\\
            {\det}(B) &: (\phi_1, \ldots, \phi_p, v_1, \ldots, v_q) \longmapsto \det( v_{b_1}, \ldots, v_{b_n}).
            \end{split}
        \end{align}
        A linear basis for $\C[V^{* p} \oplus V^q]^{\SL(V)}$ is given by the set
        \begin{align*}
        \mathcal{S}^{\SL(V)} &\coloneqq \bigcup_{\substack{\lambda : \\ \ell(\lambda) \leq n}} \left( \bigcup_{r \geq 0} \SSYT(\lambda + r^n, p) \times \SSYT(\lambda,q) \right) \cup \left( \bigcup_{s > 0} \SSYT(\lambda, p) \times \SSYT(\lambda + s^n,q) \right) \\
        &= \hspace{2ex}\bigcup_{\mathclap{(T,U) \in \mathcal{S}^{\GL(V)}}} \hspace{3ex} \Big\{(A_1 \cdots A_r T, \: U) : r \geq 0\Big\} \cup \Big\{ (T, \: B_1 \cdots B_s U) : s > 0 \Big\},
        \end{align*}
        where we identify bitableaux with elements of $\C[V^{*p} \oplus V^q]^{\SL(V)}$ as follows:
        \begin{align*}
        (A_1 \cdots A_r T, \: U) &= {\det}^*(A_1) \cdots {\det}^*(A_r) \prod_{\ell=1}^{\lambda_1} {\rm m}(T_\ell, U_\ell),\\
        (T, \: B_1 \cdots B_s U) &= {\det}(B_1) \; \cdots \; {\det}(B_s) \, \prod_{\ell=1}^{\lambda_1} {\rm m}(T_\ell, U_\ell).
        \end{align*}

        \item \label{subprop:standard monomials SL SO b} Let $G = \SO(V)$.
        Let the contractions $f_{ij}$ be as defined in Table~\ref{table:FT} for $\O(V)$, and let the minors ${\rm m}(I,J)$ be as defined in~\eqref{minor def}.
        For $A = \{a_1, \ldots, a_n\} \subseteq [m]$, define the function
        \begin{equation}
            \label{det SO}
            \det(A) : (v_1, \ldots, v_m) \longmapsto \det(v_{a_1}, \ldots, v_{a_n}).
        \end{equation}
        A linear basis for $\C[V^m]^{\SO(V)}$ is given by the set
        \begin{align*}
        \mathcal{S}^{\SO(V)} &\coloneqq \bigcup_{\mathclap{\substack{\lambda: \\ \ell(\lambda) \leq n, \\ \textup{even rows}}}} \Big(\SSYT(\lambda, m) \cup \SSYT(\lambda + 1^n, m) \Big) \\
        &= \hspace{.5ex} \bigcup_{\mathclap{T \in \mathcal{S}^{\O(V)}}} \hspace{2ex} \{T\} \cup \{AT : |A| = n \text{ and }A \leq T_1 \},
        \end{align*}
        where we identify tableaux with elements of $\C[V^m]^{\SO(V)}$ as follows:
        \begin{align*}
        T &= \prod_{\ell=1}^{\lambda_1 / 2} {\rm m}(T_{2\ell-1}, T_{2\ell}),\\
        AT & = 
        \det(A) \prod_{\ell=1}^{\lambda_1 / 2} {\rm m}(T_{2\ell-1}, T_{2\ell}).
        \end{align*}
    \end{enumerate}
    For both $G = \SL(V)$ and $\SO(V)$, a linear basis for $\C[W]^G_{\mathbf{d}}$ is given by the subset
    \[
        \mathcal{S}^{G}_{\mathbf{d}} \coloneqq \{ \textup{(bi)tableaux in $\mathcal{S}^{G}$ with weight $\mathbf{d}$}\}.
    \]
\end{prop}

\begin{proof}
    These standard monomial bases are given in~\cite{Lakshmibai}, as follows.
    For $G = \SL(V)$, the standard monomials are defined in Definition 11.3.0.1 (see also the more explicit list given in~\cite{LinshawSong} immediately following Theorem 1.1), and are shown to furnish a basis for $\C[W]^G$ in Theorem 11.3.5.3.
    For $G = \SO(V)$, the standard monomials are defined in Definition 12.2.1.1, and are shown to furnish a basis for $\C[W]^G$ in Theorem 12.3.3.2.\footnote{Our notation in this paper aligns with the notation in~\cite{Lakshmibai} as follows (see especially pages 141 and 167).
    When $G = \SL(V)$, we use $A$ and $B$ to denote an $n$-element subset of $[p]$ and of $[q]$, respectively; in~\cite{Lakshmibai}, however, the authors use $J$ and $I$, respectively.
    Therefore, where we write $\det^*(A)$ and $\det(B)$, the authors of~\cite{Lakshmibai} write $\xi(J)$ and $u(I)$, respectively.
    When $G = \SO(V)$, we use $A$ to denote an $n$-element subset of $[m]$, whereas the authors of~\cite{Lakshmibai} use $I$;
    therefore, where we write $\det(A)$, the authors of~\cite{Lakshmibai} write $u(I)$.}
    Note that although Definition 12.2.1.1 in~\cite{Lakshmibai} appears to allow arbitrarily many $\det(A)$ factors, it follows from Lemma 12.2.0.1 (i.e., $\det(A) \det(A') = {\rm m}(A,A')$) that every standard monomial has a unique expression containing at most one $\det(A)$ factor.
    It is straightforward to observe that the partial order given in the aforementioned definitions in~\cite{Lakshmibai} is equivalent to the tableau order~\eqref{tableau order}, thus allowing us to view $\mathcal{S}^{G}$ as the set of (bi)tableaux we define in Proposition~\ref{prop:standard monomials SL SO}.
    
    To verify that the subset $\mathcal{S}^{G}_{\mathbf{d}}$ is a linear basis for the component $\C[W]^G_{\mathbf{d}}$, we first recall the proof of this fact for $\GL(V)$ and $\O(V)$ in Proposition~\ref{prop:standard monomial basis}, where we determined the multidegrees of the generators $f_{ij}$.
    We now do the same for the det factors: for $\SL(V)$, the function $\det^*(A)$ is homogeneous of multidegree $(\sum_{a \in A} \mathbf{e}_{a}, \mathbf{0})$ and the function $\det(B)$ is homogeneous of multidegree $(\mathbf{0}, \sum_{b \in B} \mathbf{e}_{b})$.
    For $\O(V)$, the function $\det(A)$ is homogeneous of multidegree $\sum_{a \in A} \mathbf{e}_{a}$.
    Therefore the weight of a (bi)tableau in $\mathcal{S}^G$ equals the multidegree of its corresponding function in $\C[W]^G$.
\end{proof}

\section{RSK correspondences and ordinary monomials}
\label{sec:RSK ordinary monomials}

In this section, we use RSK-type correspondences to obtain another linear basis for $\C[W]^G_{\mathbf{d}}$ from the basis $\mathcal{S}^G_{\mathbf{d}}$ of standard monomials constructed in Section~\ref{sec:CIT standard monomials}.
We will denote this new basis by $\mathcal{B}^G_{\mathbf{d}}$, and we call its elements \emph{ordinary monomials}, in order to emphasize that they are true monomials in the generators $f_{ij}$ and $\det^*(A)$ and $\det(B)$.

\subsection*{RSK correspondences}

The Robinson--Schensted--Knuth (RSK) correspondence is well known for its remarkable properties and its ubiquity in algebraic combinatorics.
In Proposition~\ref{prop:RSK} below, we recall three variations (corresponding to the three groups $\GL(V)$, $\O(V)$, and $\Sp(V)$).
The first, which we call $\RSK_{\rm A}$, is the correspondence originally defined by Knuth~\cite{Knuth}*{\S3}; the second, which we call $\RSK_{\rm B}$, was also introduced by Knuth~\cite{Knuth}*{\S5}, and is described in \cite{Burge}*{p.~22},~\cite{HerzogTrung}*{\S1}, and~\cite{Conca}*{Thm.~1.8}; the third, which we call $\RSK_{\rm C}$, is described by Knuth~\cite{Knuth}*{Thm.~4}, as well as~\cite{Burge}*{\S2} and~\cite{HerzogTrung}*{p.~25}.

Referring the reader to the references above for the details in each variant (each of which consists of an ``insertion'' algorithm and its inverse ``deletion'' algorithm), we simply give the domain and codomain in Proposition~\ref{prop:RSK} below.
Each RSK variant is a bijection between $\mathcal{S}^G$ and a set $\mathcal{M}^G$ consisting of nonnegative integer matrices.
We write ${\rm SM}_m(\mathbb{N})$ to denote the set of $m \times m$ symmetric matrices over $\mathbb{N}$.
Given a matrix $M$, we define its \emph{support} to be ${\rm supp}(M) \coloneqq \{(i,j) : m_{ij} \neq 0 \}$.
Adopting the language of contingency tables, we define the \emph{margins} of a matrix $M \in {\rm M}_{p,q}(\mathbb{N})$ to be the ordered pair $((r_1, \ldots, r_p), (c_1, \ldots, c_q))$, where $r_i$ denotes the sum of entries in row $i$ of $M$, and $c_j$ denotes the sum of entries in column $j$ of $M$.
For $M \in {\rm SM}_m(\mathbb{N})$, the row and column sums are the same, so in this case the margins of $M$ consist of the single vector $(r_1, \ldots, r_m)$.

In order to clearly describe the codomain $\mathcal{M}^G$ of each RSK variant, we let $P$ denote the set of ordered pairs $(i,j)$ indexing the generators $f_{ij}$, as given in the top row of Table~\ref{table:FT}.
We make $P$ a poset by equipping it with the partial order $\preceq$ given below in Proposition~\ref{prop:RSK}.
Recall that a subset of $P$ is said to be a \emph{chain} if its elements are pairwise comparable with respect to $\preceq$, and is said to be an \emph{antichain} if no two of its elements are comparable.
The \emph{height} (resp., \emph{width}) of a subset $S \subseteq P$ is the cardinality of the largest chain (resp., antichain) contained in $S$.

\begin{prop}
    \label{prop:RSK}
    Let $G = \GL(V)$, $\O(V)$, or $\Sp(V)$.
    Let $P$ be the following poset:
    
    \begin{center} 
\begin{tblr}{colspec={|Q[m,c]|Q[m,c]|Q[m,c]|Q[m,c]|},stretch=1.5}

\hline

$G$ & $\GL(V)$ & $\O(V)$ & $\Sp(V)$ \\ \hline[2pt]

$P = \{(i,j): \ldots \}$ & $1 \leq i \leq p, \: 1 \leq j \leq q$ & $1 \leq i \leq j \leq m$ & $1 \leq i < j \leq m$ \\ \hline

$(i,j) \preceq (i',j')$ & $i \leq i'$ and $j \leq j'$
&
$i \leq i'$ and $j \geq j'$ & $i \leq i'$ and $j \leq j'$ \\ \hline
\end{tblr}
\end{center}

    \noindent We have the following bijections, where $\mathcal{S}^{G}$ is the set defined in Proposition~\ref{prop:standard monomial basis}:
    \[
    \begin{array}{lrl}
        \RSK_{\rm A} : & \mathcal{S}^{\GL(V)} \longrightarrow & \mathcal{M}^{\GL(V)} \coloneqq \{M \in {\rm M}_{p,q}(\mathbb{N}) : \textup{${\rm supp}(M) \subseteq P$ has width $\leq n$} \}, \\
         \RSK_{\rm B} : & \mathcal{S}^{\O(V)} \longrightarrow & \mathcal{M}^{\O(V)} \coloneqq \{M \in {\rm SM}_{m}(\mathbb{N}) : \textup{$m_{ii}$ is even, ${\rm supp}(M) \cap P$ has width $\leq n$} \}, \\
        \RSK_{\rm C} : &\mathcal{S}^{\Sp(V)} \longrightarrow & \mathcal{M}^{\Sp(V)} \coloneqq \{M \in {\rm SM}_{m}(\mathbb{N}) : \textup{$m_{ii} = 0$, ${\rm supp}(M) \cap P$ has width $\leq n$} \}.
    \end{array}
    \]
    In each type, the correspondence restricts to a bijection
    \[
    \RSK_{\bullet} : \mathcal{S}^{G}_{\mathbf{d}} \longrightarrow \mathcal{M}^G_{\mathbf{d}} \coloneqq\{M \in \mathcal{M}^G : \text{$M$ has margins $\mathbf{d}$}\}.
    \]
\end{prop}

\begin{proof}
    Let $\lambda$ be the shape of the (bi)tableau input from $\mathcal{S}^{G}$, and let $M \in \mathcal{M}^G$ be the matrix that is output by $\RSK_{\bullet}$.
    To verify the width conditions on the sets $\mathcal{M}^G$, in light of the conditions on $\ell(\lambda)$ in Proposition~\ref{prop:standard monomial basis}, it suffices to establish the following:
    \begin{equation}
        \label{RSK width height}
        \begin{array}{c|rl}
           \RSK_{\rm A}  & \ell(\lambda) \hspace{-1ex} & = \text{width of } {\rm supp}(M) \subseteq P   \\ \hline
            
            \RSK_{\rm B}  & \ell(\lambda) \hspace{-1ex}& = \text{width of } {\rm supp}(M) \cap P  \\ \hline

            \RSK_{\rm C}  & \ell(\lambda)/2 \hspace{-1ex} & = \text{width of } {\rm supp}(M) \cap P
        \end{array}
    \end{equation}
    
    For $\RSK_{\rm A}$, the result~\eqref{RSK width height} 
    is given by Knuth~\cite{Knuth}, at the bottom of page 724 (using the phrase ``longest strictly decreasing subsequence'' in place of our ``antichain in $P$'').
    Note that in the present paper, where $\RSK_{\rm A}$ maps $(T,U) \mapsto M$, we are viewing $T$ as the ``recording tableau'' and $U$ as the ``insertion tableau'' (in the standard sense of~\cite{Fulton}*{Ch.~4}), so that entries of $T$ (resp., $U$) correspond to indices of rows (resp., columns) in $M$.
    
    For $\RSK_{\rm B}$, we first remark that in all of the references cited above, this variant operates on ``dual tableaux,'' which are obtained by transposing semistandard tableaux.
    Hence in order to align the results in this paper with those in the literature, it is necessary to interchange the words ``row'' and ``column.''
    With this said, the width result in~\eqref{RSK width height} is Lemma 1.1 in Herzog--Trung~\cite{HerzogTrung}.
    
    For $\RSK_{\rm C}$, which is simply the restriction of $\RSK_{\rm A}$ to pairs $(T,T)$ whose shape has even columns~\cite{Knuth}*{Thm.~4}, the result~\eqref{RSK width height} follows from the same result for $\RSK_{\rm A}$.
    In particular, we have that $\ell(\lambda)$ equals the longest antichain in ${\rm supp}(M)$, which is a pattern of nonzero entries such that any two entries lie strictly southwest/northeast of each other.
    Since $M$ is symmetric with zeros on the diagonal, and since $P$ involves only the upper-triangular entries of $M$, restricting the longest antichain to $P$ cuts its length in half.

    The fact that $\RSK_\bullet (\mathcal{S}^G_{\mathbf{d}}) = \mathcal{M}^G_{\mathbf{d}}$ follows immediately from each RSK construction:
    in $\RSK_{\rm A}$, entries in $T$ and $U$ correspond to row and column indices (respectively) in $M$, and for the other two RSK variants, each entry in $T$ corresponds to a row \emph{and} a column index in the symmetric matrix $M$.    
\end{proof}

\subsection*{Bases of ordinary monomials}

Let $G = \GL(V)$, $\O(V)$, or $\Sp(V)$.
By an \emph{ordinary monomial} in $\C[W]^G$, we mean a monic monomial in the generators $f_{ij}$ (defined in Table~\ref{table:FT}).
Given an ordinary monomial~$\boldsymbol{f}$, define its \emph{support} to be
\begin{equation}
    \label{supp}
    {\rm supp}(\boldsymbol{f}) \coloneqq \{ (i,j) \in P : f_{ij} \text{ divides } \boldsymbol{f} \},
\end{equation}
where $P$ is the poset given in Proposition~\ref{prop:RSK}.
Following Sturmfels~\cite{Sturmfels}*{p.~138} and Herzog--Trung~\cite{HerzogTrung}*{p.~25}, we define
\begin{equation}
\label{width}
    \text{width of }\boldsymbol{f} \coloneqq \text{width of }{\rm supp}(\boldsymbol{f}) \text{ in $P$}.
\end{equation}

In Proposition~\ref{prop:RSK}, each matrix $M \in \mathcal{M}^G$ can be viewed as the degree matrix of a unique ordinary monomial, via the map
\begin{equation}
\label{Phi}
    \Phi(M) \coloneqq
    \begin{cases}
        \displaystyle\prod_{(i,j) \in P} f_{ij}^{m_{ij}}, & G = \GL(V) \textup{ or } \Sp(V),\\[4ex]
        \displaystyle\prod_{(i,i) \in P} f_{ii}^{m_{ii}/2} \prod_{\substack{(i,j) \in P:\\ i < j}} f_{ij}^{m_{ij}}, & G = \O(V).
    \end{cases}
\end{equation}
Conversely, any ordinary monomial $\boldsymbol{f}$ can be viewed as a unique degree matrix $\Phi^{-1}(\boldsymbol{f})$.
By composing with the RSK correspondences, we obtain an injective map $\Phi \circ \RSK_\bullet$ from $ \mathcal{S}^G$ into the set of ordinary monomials in $\C[W]^G$.
Moreover, recalling the multidegree of each generator $f_{ij}$ from the proof of Proposition~\ref{prop:standard monomial basis}, it is clear from~\eqref{Phi} that the margins of a matrix $M$ give the multidegree of the ordinary monomial $\Phi(M)$.
Therefore $\Phi \circ \RSK_{\bullet}$ restricts to an injective map from $\mathcal{S}^G_{\mathbf{d}}$ into the set of ordinary monomials in the multigraded component $\C[W]^G_{\mathbf{d}}$.
In fact, the image of this map furnishes (another) linear basis for $\C[W]^G_{\mathbf{d}}$.

\begin{prop}
    \label{prop:ordinary monomials}
    Let $G = \GL(V)$, $\O(V)$, or $\Sp(V)$.
    A linear basis for $\C[W]^G_{\mathbf{d}}$ is given by the set
    \[
    \mathcal{B}^G_{\mathbf{d}} \coloneqq \Phi \circ \RSK_\bullet(\mathcal{S}^G_\mathbf{d}).
    \]
\end{prop}

\begin{proof}
    We first observe that
    \begin{equation}
    \label{claim in ordinary proof}
    \mathcal{B}^{G}_{\mathbf{d}} = \{ \textup{ordinary monomials in $\C[W]^G_{\mathbf{d}}$ having width at most $n$}\},
    \end{equation}
    where width is defined as in~\eqref{width}.
    To see this, note from~\eqref{Phi} that for $M \in \mathcal{M}^G$, we have ${\rm supp}(M) \cap P = {\rm supp}(\Phi(M))$, and therefore the width of ${\rm supp}(M) \cap P$ equals the width of the ordinary monomial $\Phi(M)$.
    Hence every element of $\mathcal{B}^G_{\mathbf{d}}$ has width at most $n$, in the sense of~\eqref{width}.
    Conversely, any ordinary monomial in $\C[W]^G_{\mathbf{d}}$ of width at most $n$ corresponds via $\Phi^{-1}$ to its degree matrix in $\mathcal{M}^G_{\mathbf{d}}$, which establishes the equality in~\eqref{claim in ordinary proof}.

    The fact that the right-hand side of~\eqref{claim in ordinary proof} furnishes a linear basis for $\C[W]^G_{\mathbf{d}}$ is a straightforward consequence of~\cite{Sturmfels}*{Cor.~7} for $\GL(V)$, and is briefly argued in~\cite{HerzogTrung}*{pp.~25--26} for $\Sp(V)$, in the context of the determinantal rings on the left-hand side of~\eqref{quotient}.
    We provide a self-contained proof below for all three groups.
    
    Recall from Proposition~\ref{prop:FFT SFT GL O Sp} that $\mathcal{R}$ is a Gr\"obner basis for $\langle \mathcal{R} \rangle$.
    From the general theory of Gr\"obner bases, a linear basis for $\C[\{x_{ij}\}] / \langle \mathcal{R} \rangle$ is given by the set of monomials (in the $x_{ij}$'s) that do not belong to the ideal ${\rm init}(\mathcal{R})$ generated by the leading monomials of elements in $\mathcal{R}$.
    
    We claim that with respect to the monomial ordering induced by the variable order in Table~\ref{table:FT}, the image (under $x_{ij} \mapsto f_{ij}$) of the leading monomial in each element of $\mathcal{R}$ is a monomial of width $n+1$, in the sense of~\eqref{width}.
    First we verify the claim for $G = \GL(V)$, where each element of $\mathcal{R}$ is a minor determined by, say, row indices $i_1 < \cdots < i_{n+1}$ and column indices $j_1 > \cdots > j_{n+1}$.
    The image of the leading monomial of this minor is the product
    $f_{i_1, j_1} \cdots f_{i_{n+1}, j_{n+1}}$,
    whose support $\{(i_1, j_1), \ldots, (i_{n+1}, j_{n+1})\}$ is an antichain in $P$ of size $n+1$, and thus has width $n+1$.
    For $G = \O(V)$, each element of $\mathcal{R}$ is an $(n+1)$-minor determined by row indices $i_1 < \cdots < i_{n+1}$ and column indices $j_1 < \cdots < j_{n+1}$, and the image of its leading monomial is the product $f_{i_1, j_1} \cdots f_{i_{n+1}, j_{n+1}}$, whose support is likewise an antichain in $P$ of size $n+1$.
    For $\Sp(V)$, each element of $\mathcal{R}$ is a $2(n+1)$-Pfaffian determined by row and column indices $i_1 < \cdots < i_{2(n+1)}$.
    The image of the leading monomial of this Pfaffian is 
    \[
    f_{i_1, i_{2n+2}} f_{i_2, i_{2n+1}} \cdots f_{i_{n+1}, i_{n+2}},
    \]
    whose support is an antichain in $P$ of size $n+1$.
    
    Conversely, by the argument above, any ordinary monomial of width $n+1$ is the image (under the map $x_{ij} \mapsto f_{ij}$) of the leading monomial of some element of $\mathcal{R}$.
    Note that multiplying a monomial by another monomial cannot decrease its width.
    It follows that an ordinary monomial $\boldsymbol{f}$ belongs to the image (under $x_{ij} \mapsto f_{ij}$) of ${\rm init}(\mathcal{R})$ if and only if the width of $\boldsymbol{f}$ is strictly greater than $n$.
    Therefore, a linear basis for $\C[W]^G$ is given by the ordinary monomials of width at most $n$.   
\end{proof}

We now turn to the groups $G = \SL(V)$ and $\SO(V)$.
By an \emph{ordinary monomial} in $\C[W]^G$, we mean a monic monomial in the $f_{ij}$'s and $\det^*(A)$'s and $\det(B)$'s, where the det functions are defined above in~\eqref{det SL} and~\eqref{det SO}.
In Definition~\ref{def: M SL SO} below, we define an analogous set $\mathcal{B}^G_{\mathbf{d}}$ consisting of certain ordinary monomials obtained from $\mathcal{S}^G_{\mathbf{d}}$ via the bijection $\Phi \circ \RSK_\bullet$ given in Proposition~\ref{prop:RSK}.

\begin{dfn}
\label{def: M SL SO}
    Let $G = \SL(V)$ or $\SO(V)$.
    Define
    \[
    \mathcal{B}^G_{\mathbf{d}} \coloneqq \Phi(\mathcal{S}^G_{\mathbf{d}}),\]
    where $\Psi$ is the bijection defined as follows:

    \begin{enumerate}[label=\textup{(\alph*)}]
    \item For $G = \SL(V)$, we define
    \begin{align*}
        \Psi(A_1 \cdots A_r T, \: U) & \coloneqq {\det}^*(A_1) \cdots {\det}^*(A_r) \cdot \Phi \circ \RSK_{\rm A}(T,U), \\
        \Psi(T, \: B_1 \cdots B_s U) & \coloneqq \: {\det}(B_1) \; \cdots \; {\det}(B_s) \, \cdot \Phi \circ \RSK_{\rm A}(T,U).
    \end{align*}

    \item For $G = \SO(V)$, we define
    \begin{align*}
        \Psi(T) & \coloneqq \Phi \circ \RSK_{\rm B}(T), \\
        \Psi(AT) & \coloneqq \det(A) \cdot \Phi \circ \RSK_{\rm B}(T).
    \end{align*}
       
\end{enumerate}

For either group, define $\mathcal{B}^G \coloneqq \bigcup_{\mathbf{d}} \mathcal{B}^G_{\mathbf{d}}$.
    
\end{dfn}

The aim of the remainder of this section is to show that for $G = \SL(V)$ or $\SO(V)$, the set $\mathcal{B}^G_{\mathbf{d}}$ in Definition~\ref{def: M SL SO} is a linear basis for $\C[W]^G_{\mathbf{d}}$; see Proposition~\ref{prop:ordinary monomial basis SL SO} at the end of the section.

\begin{lemma}\
    \label{lemma:tk uk}

    \begin{enumerate}[label=\textup{(\alph*)}]
        \item \label{sublemma:tk uk a} Let $(T,U) \in \mathcal{S}^{\GL(V)}$, where the first columns of $T$ and $U$ are given by $T_1 = \{t_1, \ldots, t_\ell\}$ and $U_1 = \{u_1, \ldots, u_\ell\}$, with the elements written in increasing order.
        For each $1 \leq k \leq \ell$, the entry $t_k$ (resp., $u_k$) is the smallest integer such that ${\rm supp}(\Phi \circ \RSK_{\rm A}(T,U))$ contains an antichain $\{(i_1, j_1), \ldots, (i_k, j_k)\}$ of the following form (respectively):
        \[
        \fbox{$\begin{aligned}
            &i_1 < \cdots < i_k = t_k,  \\
            &j_{1} > \cdots > j_{k}
        \end{aligned}$} 
        \quad \textup{or}
        \quad
        \fbox{$\begin{aligned}
            &i_1 > \cdots > i_k,  \\
            &j_{1} < \cdots < j_{k} = u_k
        \end{aligned}$}
        \]

        \item \label{sublemma:tk uk b} Let $T \in \mathcal{S}^{\O(V)}$, where the first column of $T$ is given by $T_1 = \{t_1, \ldots, t_\ell\}$, with the elements written in increasing order.
        For each $1 \leq k \leq \ell$, the entry $t_k$ is the smallest integer such that ${\rm supp}(\Phi \circ \RSK_{\rm B}(T))$ contains an antichain $\{(i_1, j_1), \ldots, (i_k, j_k)\}$ of the following form:
        \[
        \fbox{$\begin{aligned}
            &i_1 < \cdots < i_k = t_k,  \\
            &j_{1} < \cdots < j_{k}
        \end{aligned}$}
        \]
    \end{enumerate}
\end{lemma}

\begin{proof}

    Part (b) is proved by Herzog--Trung~\cite{HerzogTrung} in Lemma~1.2.
    Upon applying their argument to the $\RSK_{\rm A}$ correspondence in Knuth~\cite{Knuth}*{\S3}, part (a) follows immediately.
\end{proof}

\begin{lemma}
\label{lemma:width n+1 SL SO}
    Let $\boldsymbol{f}$ be a monic monomial in the $f_{ij}$\!'s, with ${\rm supp}(\boldsymbol{f})$ as defined in~\eqref{supp}.

    \begin{enumerate}[label=\textup{(\alph*)}]
        \item Let $G = \SL(V)$.
        Suppose that $A_1 \leq \cdots \leq A_r$ and $B_1 \leq \cdots \leq B_s$ in the tableau order~\eqref{tableau order}.
        Let $A_r = \{a_1, \ldots, a_n\}$ and $B_s = \{b_1, \ldots, b_n\}$; however, if $r=0$ or $s=0$, then put $A_r = \varnothing$ or $B_s = \varnothing$, respectively.
        Set $a_{n+1} = b_{n+1} = \infty$.
        An ordinary monomial
        \[
        {\det}^*(A_1) \cdots {\det}^*(A_r) \cdot \boldsymbol{f} \qquad \text{or} \qquad \det(B_1) \cdots \det(B_s) \cdot \boldsymbol{f}
        \]
        belongs to $\mathcal{B}^{\SL(V)}$ if and only if there does not exist a pattern
        \begin{equation}
        \label{bad pattern SL}
        \fbox{$\begin{aligned}
            &i_1 < \cdots < i_k < a_{k},  \\
            &j_{1} > \cdots > j_{k}
        \end{aligned}$} 
        \qquad \textup{or}
        \qquad
        \fbox{$\begin{aligned}
            &i_1 > \cdots > i_k,  \\
            &j_{1} < \cdots < j_{k} < b_k
        \end{aligned}$} \, ,
        \end{equation}
        respectively, where $\{(i_1, j_1), \ldots, (i_k, j_k)\} \subseteq {\rm supp}(\boldsymbol{f})$.

        \item Let $G = \SO(V)$.
        Let $A = \varnothing$ or $A = \{a_1, \ldots, a_n\}$, and set $a_{n+1} = \infty$;
        corresponding to these two choices for $A$ (respectively), an ordinary monomial
        \[
        \boldsymbol{f} \quad \textup{or} \quad {\det}(A) \cdot \boldsymbol{f}
        \]
        belongs to $\mathcal{B}^{\SO(V)}$ if and only if there does not exist a pattern
        \begin{equation}
        \label{bad pattern SO}
        \fbox{
        $
        \begin{aligned}
            & i_1 < \cdots < i_k < a_k,\\
            & j_{1} < \cdots < j_{k} 
        \end{aligned}
        $
        }
        \end{equation}
        where $\{(i_1, j_1), \ldots, (i_k, j_k)\} \subseteq {\rm supp}(\boldsymbol{f})$.
    \end{enumerate}
\end{lemma}

\begin{proof}\

(a) Let $(T,U) = (\Phi \circ \RSK_{\rm A})^{-1}(\boldsymbol{f})$, and let $T_1 = \{t_1, \ldots, t_\ell\}$ and $U_1 = \{u_1, \ldots, u_\ell\}$.
By Definition~\ref{def: M SL SO}, the monomial $\det^*(A_1) \cdots \det^*(A_r) \cdot \boldsymbol{f}$ belongs to $\mathcal{B}^{\SL(V)}$ if and only if $A_r \leq T_1$ in the tableau order~\eqref{tableau order}, if and only if $a_k \leq t_k$ for all $ 1\leq k \leq \ell \leq n$.
By Lemma~\ref{lemma:tk uk}(a), we have $a_k \leq t_k$ if and only if $a_k \leq i_k$ for every antichain $\{(i_1, j_1), \ldots, (i_k, j_k)\} \subseteq {\rm supp}(\boldsymbol{f})$ where $i_1 < \cdots < i_k$.
Moreover, we have $\ell \leq n$ if and only if there does not exist the pattern~\eqref{bad pattern SL} where $k = n+1$.
The proof is identical upon replacing the $\det^*(A_i)$ factors by $\det(B_i)$ factors.

(b) The proof is identical to part (a), this time appealing to part \ref{sublemma:tk uk b} of Lemma~\ref{lemma:tk uk}.    
\end{proof}

\begin{lemma}
    \label{lemma:Sigma for SL SO}
    The set $\mathcal{B}^G$ (defined in Definition~\ref{def: M SL SO}) consists of precisely those ordinary monomials that are not divisible by ordinary monomials $\boldsymbol{m}$ of the following form:

    \begin{enumerate}[label=\textup{(\alph*)}]
        \item \label{sublemma:Sigma a} For $G = \SL(V)$:
        \begin{enumerate}[label=\textup{(\roman*)}]
            \item $\det^*(A) \det(B)$;
            \item $\det^*(A) \det^*(A')$ or $\det(B) \det(B')$, where $A$ and $A'$ (resp., $B$ and $B'$) are not comparable in the tableau order~\eqref{tableau order};
            \item $f_{i_1,j_1} \cdots f_{i_k, j_k}$ or $\det^*(A) \cdot f_{i_1, j_1} \cdots f_{i_k, j_k}$ or $\det(B) \cdot f_{i_1, j_1} \cdots f_{i_k, j_k}$ satisfying~\eqref{bad pattern SL}.
        \end{enumerate}

        \item \label{sublemma:Sigma b} For $G = \SO(V)$:

        \begin{enumerate}[label=\textup{(\roman*)}]
            \item $\det(A) \det(A')$;
            \item $f_{i_1,j_1} \cdots f_{i_k,j_k}$ or $\det(A) \cdot f_{i_1,j_1} \cdots f_{i_k,j_k}$ satisfying~\eqref{bad pattern SO}.
        \end{enumerate}
    \end{enumerate}
\end{lemma}

\begin{proof}
    To verify part (i) for each group, observe (by Proposition~\ref{prop:standard monomials SL SO} and Definition~\ref{def: M SL SO}) that no element of $\mathcal{B}^{\SL(V)}$ contains both $\det$ and $\det^*$, and no element of $\SO(V)$ contains more than one $\det$ factor.
    To verify part (a)(ii), we further observe that since an element in $\mathcal{B}^{\SL(V)}$ arises from some pair $(A_1 \cdots A_rT, \: U)$ or $(T, \: B_1 \cdots B_s U)$ of semistandard tableaux, the $A_i$'s or $B_i$'s form a chain in the tableau order.
    Parts (a)(iii) and (b)(ii) follow immediately from Lemma~\ref{lemma:width n+1 SL SO}.
\end{proof}

\begin{rem}
    For $G = \SL(V)$, our Lemma~\ref{lemma:Sigma for SL SO} is implicit in~\cite{Jackson}*{Thm.~3.6.17}, using different language, if one views $\SL(V)$-invariants as certain $\GL(V)$-covariants.
    In particular, our pattern~\eqref{bad pattern SL} is a special case of a \emph{split}~\cite{Jackson}*{p.~77}.
    Likewise, the straightening law~\eqref{Jackson straightening} below is a special case of~\cite{Jackson}*{equation (3.8.21)}.
\end{rem}

\begin{lemma}
    \label{lemma:span}
    The set $\mathcal{B}^G$ (defined in Definition~\ref{def: M SL SO}) spans $\C[W]^G$.
\end{lemma}

\begin{proof}
    We will impose a certain monomial ordering $<$ on the set of all ordinary monomials.
    Then we will show that each ordinary monomial $\boldsymbol{m}$ listed in Lemma~\ref{lemma:Sigma for SL SO} can be written as a linear combination of ordinary monomials that are less than $\boldsymbol{m}$.
    It follows that \emph{any} ordinary monomial outside $\mathcal{B}^G$ can be written as a linear combination of lesser ordinary monomials.
    Proceeding by induction, eventually one obtains a linear combination of monomials not divisible by any $\boldsymbol{m}$'s in Lemma~\ref{lemma:Sigma for SL SO}, which is therefore in the span of $\mathcal{B}^G$.

    \noindent (a) Let $G = \SL(V)$.
    First we impose a monomial ordering on $\C[\{\det^*(A)\}]$, $\C[\{\det(B)\}]$, and $\C[\{f_{ij}\}]$ separately:
   
   \begin{itemize}
    \item Viewing $A$ and $A'$ as tableau columns of length $n$, we declare that as variables, $\det^*(A) < \det^*(A')$ if the bottommost entry in $A$ is less than the bottommost entry in $A'$ (breaking ties by moving upward in the columns).
    In particular, if $A \leq A'$ in the tableau order~\eqref{tableau order}, then necessarily $\det^*(A) < \det^*(A')$.
    To obtain a monomial ordering on $\C[\{\det^*(A)\}]$, we impose the degree reverse lexicographic order induced from this variable order:
    that is, supposing $\boldsymbol{m}$ and $\boldsymbol{n}$ are monomials in $\C[\{\det^*(A)\}]$, one has $\boldsymbol{m} < \boldsymbol{n}$ if the degree of $\boldsymbol{m}$ is less than the degree of $\boldsymbol{n}$, or (if the degrees are equal) if $\boldsymbol{m}$ contains a larger power of the smallest variable (breaking ties in increasing variable order).
    
    \item The monomial ordering on $\C[\{\det(B)\}]$ is defined identically as that for $\C[\{\det^*(A)\}]$.
    
    \item The monomial ordering on $\C[\{f_{ij}\}]$ is the lexicographic order induced by the variable order in Table~\ref{table:FT}, under the map $x_{ij} \mapsto f_{ij}$.

\end{itemize}
To extend these three separate monomial orderings to a monomial ordering $<$ on the set of all ordinary monomials, we first compare the $\det^*(A)$ factors, then (to break ties) the $\det(B)$ factors, then (to break ties) the $f_{ij}$ factors.

Now we will show that each monomial $\boldsymbol{m}$ listed in Lemma~\ref{lemma:Sigma for SL SO}\ref{sublemma:Sigma a} can be written as a linear combination of lesser monomials.
We do this via the following straightening laws:

\begin{enumerate}[label=(\roman*)]
    \item Let $\boldsymbol{m} = \det^*(A) \det(B)$.
    By Theorem 11.3.1.1(i) in~\cite{Lakshmibai}, we have $\boldsymbol{m} = \det^*(A) \det(B) = {\rm m}(A,B)$.
    Since ${\rm m}(A,B)$ is a linear combination of terms in the $f_{ij}$'s only, it has degree 0 in the $\det^*$ factors, and is therefore less than $\boldsymbol{m}$ (with respect to the monomial ordering $<$).

    \item Let $\boldsymbol{m} = \det(B) \det(B')$, where $B$ and $B'$ are not comparable in the tableau order.
    By Theorem 11.3.1.1(ii) in~\cite{Lakshmibai}, $\boldsymbol{m}$ can be written as a linear combination of terms $\det(B_i)\det(B'_i)$, where (in the tableau order) we have $B_i \leq B$ and $B_i \leq B'$ and $B \leq B'_i$ and $B' \leq B'_i$.
    This implies that $\det(B_i) \leq \det(B)$ and $\det(B_i) \leq \det(B')$ with respect to the variable order.
    Therefore, with respect to the monomial ordering~$<$, we have that each $\det(B_i) \det(B'_i) \leq \det(B) \det(B')$.
    Thus $\boldsymbol{m} = \det(B)\det(B')$ can be written as a linear combination of lesser monomials.
    The proof for products $\det^*(A)\det^*(A')$ is identical.

    \item If $\boldsymbol{m} = f_{i_1,j_1} \cdots f_{i_k,j_k}$ satisfies~\eqref{bad pattern SL}, then we have $k = n+1$.
    Thus $\boldsymbol{m}$ is the largest monomial in an $(n+1) \times (n+1)$ minor in $[f_{ij}]$, which is identically 0 by Proposition~\ref{prop:FFT SFT GL O Sp}.
    Thus $\boldsymbol{m}$ can be written as a linear combination of the other terms of the minor, which are less than $\boldsymbol{m}$ in the monomial ordering $<$.
    
    Suppose that $\boldsymbol{m} = \det(B) \cdot f_{i_1,j_1} \cdots f_{i_k,j_k}$ satisfies the pattern on the right-hand side of~\eqref{bad pattern SL}.
    Put $B = \{b_1, \ldots, b_n\}$ and $I = \{i_k, \ldots, i_1\}$ and $J = \{j_1, \ldots, j_k\}$, with elements written in increasing order.
    Consider the set partition $B = C \sqcup D$, where $C = \{b_1, \ldots, b_{k-1}\}$ and $D = \{b_k, \ldots, b_n\}$.
    Since we are assuming that $\boldsymbol{m}$ satisfies~\eqref{bad pattern SL}, we have
    \begin{equation}
        \label{IB}
        \underbrace{j_1 < \cdots < j_k}_{J} < \underbrace{b_k < \cdots < b_n}_D,
    \end{equation}
    so that $| J \sqcup D | = n+1$.
    Now consider all possible set partitions $J \sqcup D = J' \sqcup D'$ such that $|J'| = |J|$ and $|D'| = |D|$.
    Let $J'D'$ denote the concatenation of $J'$ and $D'$ such that the elements of $J'$ and $D'$ are (separately) written in increasing order; note that $J'D'$ is a sequence of $n+1$ many distinct elements.
    Let ${\rm sgn}(J'D')$ denote the signature of the permutation required to restore the elements of $J'D'$ to increasing order.
    Let $CD'$ denote the concatenation of $C$ and $D'$ in the same way; note that $CD'$ is a sequence of $n$ (possibly repeating) elements.
    We claim that
    \begin{equation}
    \label{Jackson straightening}
    \sum_{\mathclap{\substack{ J', D': \\
    |J'| = |J|,\\
    |D'| = |D|,\\
    J' \sqcup D' = J \sqcup D}}}{\rm sgn}(J'D') \cdot {\det}(CD') \cdot {\rm m}(I,J') = 0.
    \end{equation}
    To prove the claim, since $\dim V = n$, it suffices to show that the multilinear function on the left-hand side of~\eqref{Jackson straightening} is alternating in the $n+1$ many vectors $v_{j_1}, \ldots, v_{j_k}, v_{b_k}, \ldots, v_{b_n} \in V$.
    To this end, choose any two of these vectors and consider the effect on the summand indexed by $J',D'$ in~\eqref{Jackson straightening} achieved by interchanging these two vectors.
    If both vectors are indexed by elements of $J'$, then interchanging them has no effect on $\det(CD')$ but reverses the sign of the minor ${\rm m}(I,J')$; the opposite is true if both vectors are indexed by elements of $D'$.
    If one vector is indexed by an element $j \in J'$ and the other by an element $d \in D'$, then interchanging them results in the term
    \begin{align*}
     & \; {\rm sgn}(J' D') \cdot \det(C \tilde{D}'') \cdot {\rm m}(I, \tilde{J}'') \\
    = \; & \underbrace{{\rm sgn}(J' D') \cdot {\rm sgn}(\tilde{J}'') \cdot {\rm sgn}(\tilde{D}'')}_{- {\rm sgn}(J''D'')} \cdot \det(C D'') \cdot {\rm m}(I, J''),
    \end{align*}
    where $\tilde{J}''$ and $\tilde{D}''$ are the (not necessarily increasing) sequences obtained from $J$ and $D$ by exchanging the elements $j$ and $d$, and where $J''$ and $D''$ are obtained by sorting these into increasing order.
    (See the detailed argument in the proof of Lemma~3.8.18 and equation (3.8.21) in~\cite{Jackson}, which carefully accounts for the signatures above.)
    Therefore, interchanging any two of the $n+1$ many vectors indexed by $J \sqcup D$ changes the sign of every term on the left-hand side of~\eqref{Jackson straightening}.
    It follows that this left-hand side is alternating in $n+1$ vectors and therefore vanishes, which proves the claim~\eqref{Jackson straightening}.

    Finally, we observe that the largest monomial in~\eqref{Jackson straightening} is $\boldsymbol{m} = \det(B) \cdot f_{i_1, j_1} \cdots f_{i_k, j_k}$, corresponding to the summand where $J'= J$ and $D' = D$.
    Indeed, by~\eqref{IB}, with respect to the monomial ordering $<$ defined above, we have $\det(CD') < \det(CD) = \det(B)$ for all $D' \neq D$; then as observed in the proof of Proposition~\ref{prop:ordinary monomials}, the largest monomial in ${\rm m}(I,J)$ is the antidiagonal product $f_{i_1, j_1} \cdots f_{i_k,j_k}$.
    Hence~\eqref{Jackson straightening} exhibits a way to write $\boldsymbol{m}$ as a linear combination of lesser monomials.

    The proof for $\boldsymbol{m} = \det^*(A) \cdot f_{i_1, j_1} \cdots f_{i_k, j_k}$ satisfying the left-hand side of~\eqref{bad pattern SL} is identical.
    
\end{enumerate}

\noindent (b) Let $G = \SO(V)$.
The proof is a simpler version of the same argument in part (a) above.
We impose the same monomial ordering $<$ on the set of ordinary monomials as we did in part (a), the only difference being that the variable ordering on the $f_{ij}$'s is that given in Table~\ref{table:FT} for the group $\O(V)$.
Then we use the following straightening laws to rewrite each monomial $\boldsymbol{m}$ in Lemma~\ref{lemma:Sigma for SL SO} as a linear combination of lesser monomials:

\begin{enumerate}[label=(\roman*)]
    \item Let $\boldsymbol{m} = \det(A) \det(A')$.
    By Lemma 12.2.0.1 in~\cite{Lakshmibai}, we have $\boldsymbol{m} = \det(A) \det(A') = {\rm m}(A,A')$.
    Since ${\rm m}(A,A')$ has degree 0 in the $\det$ variables, every term is less than $\boldsymbol{m}$ with respect to the monomial ordering $<$.

    \item Just as for part (a), we can dispense with the case where $\boldsymbol{m} = f_{i_1, j_1} \cdots f_{i_k, j_k}$ satisfies~\eqref{bad pattern SO}, since then $\boldsymbol{m}$ is the leading term of an $(n+1) \times (n+1)$ minor of $[f_{ij}]$.
    Therefore, suppose that $\boldsymbol{m} = \det(A) \cdot f_{i_1, j_1} \cdots f_{i_k, j_k}$ satisfies~\eqref{bad pattern SO}.
    Just as above, we put $I = \{i_1, \ldots, i_k\}$, $J = \{j_1, \ldots, j_k\}$, $A = \{a_1, \ldots, a_n\}$, and the set partition $A = C \sqcup D$ where $C = \{a_1, \ldots, a_{k-1}\}$ and $D = \{a_k, \ldots, a_n\}$.
    By considering the analogue of~\eqref{Jackson straightening} given by
    \[
    \sum_{\mathclap{\substack{ I', D': \\
    |I'| = |I|,\\
    |D'| = |D|,\\
    I' \sqcup D' = I \sqcup D}}}{\rm sgn}(I'D') \cdot {\det}(CD') \cdot {\rm m}(I',J) = 0,
    \]
    we observe that $\boldsymbol{m}$ is the largest monomial in the sum, thus obtaining a straightening law in the same way as in part (a)(iii) above.
    \qedhere
\end{enumerate}
\end{proof}

\begin{prop}
    \label{prop:ordinary monomial basis SL SO}

    Let $G = \SL(V)$ or $\SO(V)$, with $\mathcal{B}^G_{\mathbf{d}}$ as in Definition~\ref{def: M SL SO}.
    The set $\mathcal{B}^G_{\mathbf{d}}$ is a linear basis for $\C[W]^G_{\mathbf{d}}$.
\end{prop}

\begin{proof}
By Proposition~\ref{prop:standard monomials SL SO}, the set $\mathcal{S}^G_{\mathbf{d}}$ is a linear basis for $\C[W]^G_{\mathbf{d}}$.
Because the map $\Psi: \mathcal{S}^G_{\mathbf{d}} \rightarrow \mathcal{B}^G_{\mathbf{d}}$ (defined in Definition~\ref{def: M SL SO}) is a bijection, we have $\# \mathcal{B}^{G}_{\mathbf{d}} = \#\mathcal{S}^G_\mathbf{d} = \dim \C[W]^{G}_{\mathbf{d}}$.
Since $\mathcal{B}^G$ spans $\C[W]^G$ by Lemma~\ref{lemma:span}, and since there are exactly enough elements in each $\mathcal{B}^G_{\mathbf{d}}$ to match the dimension of the multigraded component $\C[W]^G_{\mathbf{d}}$, the result follows.
\end{proof}

\section{Linear bases for polynomial invariants via arc diagrams}
\label{sec:polynomial bases}

In this section, we use arc diagrams with degree sequence $\mathbf{d}$ to give a combinatorial interpretation of the ordinary monomials in the basis $\mathcal{B}^{G}_{\mathbf{d}}$ for $\C[W]^G_{\mathbf{d}}$.
The arcs in each diagram represent the generators $f_{ij}$ given in Table~\ref{table:FT}, and certain hyperedges represent the factors $\det^*(A)$ and $\det(B)$.
The corresponding monomial in $\mathcal{B}^G_{\mathbf{d}}$ is recovered simply by taking the product of the arcs and hyperedges in a diagram.

An \emph{arc diagram} is a multigraph with self-loops on a vertex set $[m] \coloneqq \{1, \ldots, m\}$, where the vertices are drawn from left to right in a horizontal row, and all arcs are depicted above the vertices.
As usual, the \emph{degree} of a vertex is the number of edges incident to it (where a self-loop contributes 2 to the degree).
We write the \emph{degree sequence} as an $m$-tuple $\mathbf{d} = (d_1, \ldots, d_m) \in \mathbb{N}^m$, where $d_i$ is the degree of the vertex $i$.
A \emph{$(p,q)$-bipartite} arc diagram has the vertex set $\{1^*, \ldots, p^*, 1, \ldots, q\}$, and every arc takes the form $\{i^*, j\}$ for some $i \in [p]$ and $j \in [q]$.
We depict such a diagram with a vertical line separating the $p$ starred vertices from the $q$ unstarred vertices.
Its degree sequence is thus written as an ordered pair $(\mathbf{d}_p, \mathbf{d}_q) \in \mathbb{N}^p \times \mathbb{N}^q$.
Two arcs are said to be a \emph{strict} (resp., \emph{weak}) \emph{nesting} if one arc lies strictly (resp., weakly) inside the other; here, ``strict'' and ``weak'' refer to whether or not the arcs are allowed to share vertices.
A strict \emph{$k$-nesting} is a set of $k$ many arcs in which each pair is a strict nesting.
Vacuously, every arc is a $1$-nesting.
See Figure~\ref{fig:poly arc GL O Sp} for examples of the arc diagrams described in the following theorem.

\begin{theorem} Let $G = \GL(V)$, $\O(V)$, or $\Sp(V)$.
Recall the generators $f_{ij}$ given in Table~\ref{table:FT}.
\label{thm:poly invariants GL O Sp}

    \begin{enumerate}[label=\textup{(\alph*)}]
        \item \label{GL in poly thm} Let $\dim V = n$.
        A linear basis for $\C[V^{*p} \oplus V^q]^{\GL(V)}_{(\mathbf{d}_p, \mathbf{d}_q)}$ is given by the set of $(p,q)$-bipartite arc diagrams with degree sequence $(\mathbf{d}_p, \mathbf{d}_q)$ that do not contain a strict $(n+1)$-nesting.
        An arc diagram is viewed as the product of its arcs, where each arc $\{i^*, j\}$ represents $f_{ij}$.

        \item \label{O in poly thm} Let $\dim V = n$.
        A linear basis for $\C[V^m]^{\O(V)}_{\mathbf{d}}$ is given by the set of arc diagrams with degree sequence $\mathbf{d}$, such that every collection of $n+1$ many arcs contains a weakly nested pair.
        (Equivalently, the arcs can be decomposed into at most $n$ many disjoint weak nestings.)
        An arc diagram is viewed as the product of its arcs, where each arc $\{i, j\}$ represents $f_{ij}$.

        \item \label{Sp in poly thm} Let $\dim V = 2n$.
        A linear basis for $\C[V^m]^{\Sp(V)}_{\mathbf{d}}$ is given by the set of arc diagrams without self-loops and with degree sequence $\mathbf{d}$ that do not contain a strict $(n+1)$-nesting.
        An arc diagram is viewed as the product of its arcs, where each arc $\{i, j\}$ represents $f_{ij}$.
    \end{enumerate}
\end{theorem}

\begin{figure}[t]
     \centering
     \begin{subfigure}{0.3\textwidth}
         \centering
         \scalebox{0.7}{
\begin{tikzpicture}[-,auto,node distance=.75cm,
  very thick,plain node/.style={minimum size=0.6cm,inner sep=0pt, circle,draw,font=\sffamily\bfseries, fill=white}, bend left = 60]

\node[plain node] (1*) {1*};
\node[plain node] (2*) [right of=1*] {2*};
\node[plain node] (3*) [right of=2*] {3*};
\node[plain node] (4*) [right of=3*] {4*};
\node[plain node] (5*) [right of=4*] {5*};
\node (divider) [right of=5*] {$\smash{\Bigg|}$};
\node[plain node] (1) [right of=divider] {1};
\node[plain node] (2) [right of=1] {2};
\node[plain node] (3) [right of=2] {3};
\node[plain node] (4) [right of=3] {4};

\draw [bend left=70] (1*) to (3);
\draw [bend left=80] (1*) to (3);
\draw [red, ultra thick] (1*) to (3);
\draw (2*) to (4);
\draw [red, ultra thick] (3*) to (2);
\draw [bend left=45] (3*) to (1);
\draw [bend left=55] (3*) to (1);
\draw [red, ultra thick, bend left=45] (5*) to (1);

\end{tikzpicture}
}
         \caption{$G = \GL(V)$, with $p=5$ and $q=4$.
         The degree sequence is $\mathbf{d} = ((3,1,3,0,1),(3,1,3,1))$.
         The largest strict nesting is the 3-nesting given by the arcs highlighted in red.
         Therefore this arc diagram appears as a basis element if and only if $\dim V \geq 3$.}
         \label{subfig:poly_arc_GL}
     \end{subfigure}
     \hfill
     \begin{subfigure}{0.3\textwidth}
         \centering
         \scalebox{0.7}{
\begin{tikzpicture}[-,auto,node distance=.75cm,
  very thick, inner sep=0 pt, plain node/.style={circle,draw,font=\sffamily\bfseries,fill=white,minimum size = 0.6cm}]
\tikzstyle{every loop}=[-, distance = 10]

\node[plain node] (1) {1};
\node[plain node] (2) [right of=1] {2};
\node[plain node] (3) [right of=2] {3};
\node[plain node] (4) [right of=3] {4};
\node[plain node] (5) [right of=4] {5};
\node[plain node] (6) [right of=5] {6};


\draw [red, ultra thick, bend left = 60] (1) to (3);
\draw [red, ultra thick, bend left = 80] (2) to (4);
\draw [red, ultra thick, bend left = 60] (4) to (5);
\draw [bend left = 65] (1) to (5);
\draw [bend left = 75] (1) to (5);
\draw [bend left = 65] (1) to (5);
\draw [bend left = 85] (1) to (5);
\draw [bend left = 50] (3) to (6);


\path (3) edge [loop above] node {} (3);
\path (6) edge [red, ultra thick, loop above] node {} (6);
\path (6) edge [loop above, in=55, out=125,distance=25] node {} (6);

\node at (0,-.9) {};

\end{tikzpicture}
}
         \caption{$G = \O(V)$, with ${m = 6}$.
         The degree sequence is $\mathbf{d} = (4, 1, 4, 2, 4, 5)$.
         The largest subset of arcs without a weak nesting consists of the four arcs highlighted in red; therefore this arc diagram appears as a basis element if and only if $\dim V \geq 4$.}
         \label{subfig:poly_arc_O}
     \end{subfigure}
     \hfill
     \begin{subfigure}{0.3\textwidth}
         \centering
         \scalebox{0.7}{
\begin{tikzpicture}[-,auto,node distance=.75cm,
  very thick, inner sep=0 pt, plain node/.style={circle,draw,font=\sffamily\bfseries,fill=white,minimum size = 0.6cm}]
\tikzstyle{every loop}=[-, distance = 10]

\node[plain node] (1) {1};
\node[plain node] (2) [right of=1] {2};
\node[plain node] (3) [right of=2] {3};
\node[plain node] (4) [right of=3] {4};
\node[plain node] (5) [right of=4] {5};
\node[plain node] (6) [right of=5] {6};


\draw [bend left = 55] (1) to (3);
\draw [red, ultra thick, bend left = 60] (3) to (4);
\draw [red, ultra thick, bend left = 60] (1) to (5);
\draw [bend left = 70] (1) to (5);
\draw [bend left = 80] (1) to (5);
\draw [blue, ultra thick, bend left = 70] (3) to (6);
\draw [blue, ultra thick, bend left = 60] (4) to (5);

\node at (0,-.9) {};

\end{tikzpicture}
}
         \caption{$G = \Sp(V)$, with ${m = 6}$.
         The degree sequence is $\mathbf{d} = (4,0,3,2,4,1)$.
         The largest strict nesting has size 2 (either the red arcs or the blue arcs).
         Therefore this arc diagram appears as a basis element if and only if $\frac{1}{2}\dim V \geq 2$.}
         \label{subfig:poly_arc_Sp}
     \end{subfigure}
        \caption{\label{fig:poly arc GL O Sp} Examples of the arc diagrams specified in Theorem~\ref{thm:poly invariants GL O Sp}, when $G = \GL(V)$, $\O(V)$, or $\Sp(V)$.
        Each arc diagram represents an ordinary monomial in $\mathcal{B}^G_{\mathbf{d}}$, namely the product of the contractions $f_{ij}$ corresponding to its arcs.}
    \end{figure}

\begin{proof}
    By Proposition~\ref{prop:ordinary monomials}, a linear basis for $\C[W]^G_{\mathbf{d}}$ is given by the set $\mathcal{B}^{G}_{\mathbf{d}}$;
    by~\eqref{claim in ordinary proof}, this set consists of those ordinary monomials of multidegree $\mathbf{d}$ and of width at most $n$.
    It is clear that the degree sequence of an arc diagram equals the multidegree of its corresponding monomial, because for $\GL(V)$ each generator $f_{ij}$ has multidegree $(\mathbf{e}_i, \mathbf{e}_j)$, and for $\O(V)$ and $\Sp(V)$ each $f_{ij}$ has multidegree $\mathbf{e}_i + \mathbf{e}_j$.
    It remains to prove that the nesting conditions described above are equivalent to the width being at most $n$.
    Recall from~\eqref{width} that an ordinary monomial has width at most $n$ if and only if its support in $P$ does not contain an antichain of cardinality $n+1$.

    \begin{enumerate}[label=(\alph*)]
        \item Let $G = \GL(V)$.
        With respect to the partial order $\preceq$ on $P$ given in Proposition~\ref{prop:ordinary monomials}, two elements $(i,j)$ and $(i', j')$ are noncomparable if and only if $i < i'$ and $j > j'$, or $i > i'$ and $j < j'$.
        Equivalently, the corresponding arcs $\{i^*, j\}$ and $\{i'^*, j'\}$ form a strict nesting.
        Thus an antichain in $P$ of cardinality $n+1$ corresponds to a strict $(n+1)$-nesting in an arc diagram.

        \item Let $G = \O(V)$.
        With respect to the partial order $\preceq$ on $P$, two elements $(i,j)$ and $(i', j')$ are noncomparable if and only if $i < i'$ and $j < j'$, or $i > i'$ and $j > j'$.
        Equivalently, the corresponding arcs $\{i, j\}$ and $\{i', j'\}$ do not form a weak nesting.
        Thus an antichain in $P$ of cardinality $n+1$ corresponds to a collection of $n+1$ many arcs containing no weak nesting.

        \item Let $G = \Sp(V)$.
        With respect to the partial order $\preceq$ on $P$, two elements $(i,j)$ and $(i', j')$ are noncomparable if and only if $i < i'$ and $j > j'$, or $i > i'$ and $j < j'$.
        Equivalently, the corresponding arcs $\{i, j\}$ and $\{i', j'\}$ form a strict nesting.
        Thus an antichain in $P$ of cardinality $n+1$ corresponds to a strict $(n+1)$-nesting in an arc diagram. \qedhere
    \end{enumerate}
\end{proof}

In order to extend Theorem~\ref{thm:poly invariants GL O Sp} to the groups $\SL(V)$ and $\SO(V)$, we will allow certain hyperedges of order $n = \dim V$ (i.e., edges which connect $n$ vertices).
The degree of a vertex is still defined as the number of incident (hyper)edges.
We write hyperedges as sets $A^* = \{a^*_1, \ldots, a^*_n\}$ or $B = \{b_1, \ldots, b_n\}$, with the elements written in increasing order.

\begin{theorem}
\label{thm:poly invariants SL SO}

Let $G = \SL(V)$ or $\SO(V)$.
Let $\dim V = n$.

    \begin{enumerate}[label=\textup{(\alph*)}]
        \item \label{SL in poly thm} A linear basis for $\C[V^{*p} \oplus V^q]^{\SL(V)}_{(\mathbf{d}_p, \mathbf{d}_q)}$ is given by the set of arc diagrams described for $\GL(V)$ in Theorem~\ref{thm:poly invariants GL O Sp}\ref{GL in poly thm}, where now we allow hyperedges $A^* = \{a^*_1, \ldots, a^*_n\}$ or $B = \{b_1, \ldots, b_n\}$ in the following manner:

        \begin{itemize}
            \item All hyperedges lie on the same side of the diagram as each other (either all starred or all unstarred).
            \item The hyperedges form a chain with respect to the tableau order~\eqref{tableau order}.  
            
            \item For each $1 \leq k \leq n$, the vertex $a^*_k$ (resp., $b_k$) must not lie strictly inside (resp., outside) a strict $k$-nesting.
            (Note that it suffices to check only the maximal hyperedge.)
        \end{itemize}
        
        Each arc diagram is viewed as the product of its arcs and hyperedges, where an arc $\{i^*, j\}$ represents $f_{ij}$, a hyperedge $A^*$ represents $\det^*(A)$, and a hyperedge $B$ represents $\det(B)$, as defined in Proposition~\ref{prop:standard monomials SL SO}\ref{subprop:standard monomials SL SO a}.

        \item \label{SO in poly thm} 
        A linear basis for $\C[V^m]^{\SO(V)}_{\mathbf{d}}$ is given by the set of arc diagrams described for $\O(V)$ in Theorem~\ref{thm:poly invariants GL O Sp}\ref{O in poly thm}, where now we allow at most one hyperedge $A = \{a_1, \ldots, a_n\}$ in the following manner:

        \begin{itemize}
            \item For each $1 \leq k \leq n$, consider the arcs whose left endpoint is strictly left of vertex $a_k$.
            Every collection of $k$ of these arcs must contain a weak nesting.
            (Equivalently, these arcs can be decomposed into fewer than $k$ many disjoint weak nestings.)
        \end{itemize}

        Each diagram is viewed as the product of its arcs and its hyperedge (if there is one), where each arc $\{i,j\}$ represents $f_{ij}$, and a hyperedge $A$ represents $\det(A)$, as defined in Proposition~\ref{prop:standard monomials SL SO}\ref{subprop:standard monomials SL SO b}.
    \end{enumerate}
\end{theorem}

\begin{proof}
    The result follows from Proposition~\ref{prop:ordinary monomial basis SL SO}, by realizing an ordinary monomial in $\mathcal{B}^G$ as an arc diagram in the obvious way; details are the same as in the proof of Theorem~\ref{thm:multilinear GL O Sp}.
    The conditions on the hyperedges follow immediately from Lemma~\ref{lemma:Sigma for SL SO}.
\end{proof}

    \begin{figure}[b]
        
     \begin{subfigure}[t]{0.48\textwidth}
         \centering
         \scalebox{0.7}{
\begin{tikzpicture}[-,auto,node distance=.75cm,
  very thick,plain node/.style={minimum size=0.6cm,inner sep=0pt, circle,draw,font=\sffamily\bfseries, fill=white}, painted node/.style={minimum size=0.6cm,inner sep=0pt,circle,draw,font=\sffamily\bfseries,fill=lightgray}, smaller node/.style={minimum size=0.4cm,inner sep=0pt,circle,draw,fill=lightgray}, bend left = 60]

\node[plain node] (1*) {1*};
\node[painted node] (2*) [right of=1*] {2*};
\node[plain node] (3*) [right of=2*] {3*};
\node[painted node] (4*) [right of=3*] {4*};
\node[painted node] (5*) [right of=4*] {5*};
\node[plain node] (6*) [right of=5*] {6*};
\node[plain node] (7*) [right of=6*] {7*};
\node[painted node] (8*) [right of=7*] {8*};
\node (divider) [right of=8*] {$\smash{\Bigg|}$};
\node[plain node] (1) [right of=divider] {1};
\node[plain node] (2) [right of=1] {2};
\node[plain node] (3) [right of=2] {3};
\node[plain node] (4) [right of=3] {4};
\node[plain node] (5) [right of=4] {5};

\draw [bend left=60] (2*) to (4);
\draw [bend left=70] (2*) to (4);
\draw [bend left=60] (4*) to (3);
\draw [bend left=60] (5*) to (1);
\draw [bend left=60] (7*) to (2);
\draw [bend left=70] (7*) to (2);
\draw [bend left=50] (7*) to (2);
\draw [bend left=60] (6*) to (5);

\node[smaller node] (a) [below of=1*] {};
\node[smaller node] (b) [below of=4*] {};
\node[smaller node] (c) [below of=5*] {};
\node[smaller node] (d) [below of=7*] {};

\node[smaller node] (e) [below of=a] {};
\node[smaller node] (f) [below of=e] {};
\node[smaller node] (g) [below of=f] {};
\draw[dotted] (2*) -- (a) -- (e) -- (f) -- (g);

\node[smaller node] (h) [right of=e] {};
\node[smaller node] (i) [below of=h] {};
\node[smaller node] (j) [below of=i] {};
\draw[dotted] (4*) -- (b) -- (h) -- (i) -- (j);

\node[smaller node] (k) [below of=c] {};
\node (ll) [below of=k] {};
\node[smaller node] (l) [left of =ll] {};
\node[smaller node] (m) [below of=l] {};
\draw[dotted] (5*) -- (c) -- (k) -- (l) -- (m);

\node[smaller node] (n) [right of=k] {};
\node[smaller node] (o) [below of=n] {};
\node[smaller node] (p) [right of=m] {};
\draw[dotted] (8*) -- (d) -- (n) -- (o) -- (p);

\end{tikzpicture}
}
         \caption{$G = \SL(V)$, where $n = \dim V = 4$, with $p=8$ and $q=5$.
         There are five hyperedges, depicted by the five rows of shaded dots: the topmost hyperedge is $\{2^*, 4^*, 5^*, 8^*\}$, the next is $\{1^*, 4^*, 5^*, 8^*\}$, and so on until the bottommost hyperedge $\{1^*, 2^*, 4^*, 5^*\}$.
         To verify that the hyperedges form a chain, we observe that the dotted ``tentacles'' are nonintersecting and travel weakly to the left as they descend.
         The degree sequence is $\mathbf{d} = ((4,6,0,5,5,3,4,1), (1,3,1,2,1))$.
         }
         \label{subfig:poly_arc_SL}
     \end{subfigure}
     \hfill
     \begin{subfigure}[t]{0.48\textwidth}
         \centering
         \scalebox{0.7}{
\begin{tikzpicture}[-,auto,node distance=.75cm,
  very thick, inner sep=0 pt, plain node/.style={circle,draw,font=\sffamily\bfseries,fill=white,minimum size = 0.6cm}, painted node/.style={minimum size=0.6cm,inner sep=0pt,circle,draw,font=\sffamily\bfseries,fill=lightgray}]
\tikzstyle{every loop}=[-, distance = 10]

\node[painted node] (1) {1};
\node[plain node] (2) [right of=1] {2};
\node[painted node] (3) [right of=2] {3};
\node[plain node] (4) [right of=3] {4};
\node[plain node] (5) [right of=4] {5};
\node[plain node] (6) [right of=5] {6};
\node[painted node] (7) [right of=6] {7};
\node[plain node] (8) [right of=7] {8};
\node[painted node] (9) [right of=8] {9};
\node[plain node] (10) [right of=9] {10};

\draw [bend left = 65] (2) to (4);
\draw [bend left = 75] (3) to (6);
\draw [bend left = 80] (9) to (10);
\draw [bend left = 85] (1) to (5);
\draw [bend left = 50] (3) to (6);


\path (7) edge [loop above, distance=20] node {} (7);

\node at (0,-3.2) {};

\end{tikzpicture}
}
         \caption{$G = \SO(V)$, where $n = \dim V = 4$, with $m = 10$.
         There is one hyperedge $\{1,3,7,9\}$, depicted by the four shaded vertices.
         The degree sequence is $\mathbf{d} = (1,2,3,1,1,2,3,0,2,1)$.}
         \label{subfig:poly_arc_SO}
     \end{subfigure}
        \caption{\label{fig:poly arc SL SO} Examples of the arc diagrams in Theorem~\ref{thm:poly invariants SL SO}, where $G = \SL(V)$ or $\SO(V)$.
        Each arc diagram represents an ordinary monomial in $\mathcal{B}^G_{\mathbf{d}}$, where the arcs represent contractions $f_{ij}$ and the hyperedges $A^*$ represent the functions $\det^*(A)$.}
\end{figure}

In our arc diagrams, we depict the hyperedges as follows, in order to easily verify the conditions given in Theorem~\ref{thm:poly invariants SL SO}; see Figure~\ref{fig:poly arc SL SO} for examples.
We depict a hyperedge $A = \{a_1, \ldots, a_n\}$ by a row of $n$ shaded circles, aligned horizontally with the vertices $a_1, \ldots, a_n$ (and similarly for $A^*$).
For $\SO(V)$, since there is at most one hyperedge, we simply shade the vertices $a_1, \ldots, a_n$ themselves.
For $\SL(V)$, there can be any number of hyperedges forming a chain with respect to the tableau order~\eqref{tableau order}; 
accordingly, we depict the maximal hyperedge by shading its vertices, and then we depict any additional hyperedges below the arc diagram, in decreasing order from top to bottom.
As a visual aid, for each $1 \leq k \leq n$, we draw a ``tentacle,'' consisting of dotted line segments connecting the $k$th dots in each hyperedge; as a result, the hyperedges form a chain in the tableau order if and only if these tentacles are nonintersecting and travel weakly to the left as they descend.
(See~\cite{EricksonHunzikerB} for the theory of \emph{jellyfish diagrams} in the more general context of modules of covariants.)

\section{Linear bases for tensor invariants via arc diagrams}
\label{sec:tensor bases}

Tensor (i.e., multilinear) invariants can be viewed as a special case of polynomial invariants.
As we explain below, by restricting our attention to arc diagrams which are \emph{1-regular} (i.e., every vertex has degree 1), we obtain linear bases of tensor invariants as special cases of Theorems~\ref{thm:poly invariants GL O Sp} and~\ref{thm:poly invariants SL SO}.

Let $\mathbf{1} \coloneqq (1, \ldots, 1)$.
There is a canonical isomorphism of $\GL(V)$-modules given by
\begin{align}
\label{iso GL}
\begin{split}
V^{\otimes p} \otimes V^{*\otimes q}
 &\longrightarrow \C[V^{*p} \oplus V^q]_{(\mathbf{1}, \mathbf{1})}, \\
 w_1 \otimes \cdots \otimes w_p \otimes \psi_1 \otimes \cdots \otimes \psi_q &\longmapsto \Big[ (\phi_1, \ldots, \phi_p, v_1, \ldots, v_q) \mapsto \prod_{i=1}^p \phi_i(w_i) \prod_{j=1}^q \psi_j(v_j) \Big]
 \end{split}
 \end{align}
and extended by linearity.
By Theorem~\ref{thm:poly invariants GL O Sp}\ref{GL in poly thm}, we see that $\C[V^{*p} \oplus V^q]^{\GL(V)}_{(\mathbf{1}, \mathbf{1})}$ is nonzero if and only if $p=q$, since this is the only way that a bipartite arc diagram can be 1-regular.
Therefore, from now on we will write $m$ for the common parameter $p=q$.
Note that the set of 1-regular $(m,m)$-bipartite arc diagrams --- that is, the set of $(m,m)$-bipartite \emph{perfect matchings} --- is in natural bijection with the symmetric group $\mathfrak S_m$: in particular, each element $\sigma \in \mathfrak{S}_m$ gives the arc diagram where the vertex $i^*$ is connected to the vertex $\sigma(i)$. 
By Theorem~\ref{thm:poly invariants GL O Sp}\ref{GL in poly thm}, this arc diagram represents the invariant polynomial $\prod_{i=1}^m f_{i, \sigma(i)}$, whose preimage under the isomorphism~\eqref{iso GL} is 
\begin{equation}
\label{preimage GL}
\sum_{\mathbf{c} \in [n]^m} \left(\bigotimes_{i=1}^m e_{c_i}\right) \otimes \left(\bigotimes_{j=1}^m \varepsilon_{c_{\sigma^{-1}(j)}}\right),
\end{equation}
where $\{e_1, \ldots, e_n\}$ and $\{\varepsilon_1, \ldots, \varepsilon_n\}$ are dual bases for $V$ and $V^*$, respectively.
There is an even more transparent (and coordinate-free) way to view the arc diagram as a basis element, if we view $V^{\otimes m} \otimes V^{*\otimes m}$ as its own dual space, via the natural duality map $V^{*\otimes m} \cong (V^{\otimes m})^*$:
from this perspective, the perfect matching defined by $\sigma \in \mathfrak{S}_m$, which previously represented the tensor~\eqref{preimage GL}, instead represents the functional on $V^{\otimes m} \otimes V^{*\otimes m}$ defined on simple tensors as follows (see~\cite{GW}*{\S5.3.1} for further details):
\begin{equation}
    \label{preimage GL dual}
    v_1 \otimes \cdots \otimes v_m \otimes \phi_1 \otimes \cdots \otimes \phi_m \longmapsto \prod_{i=1}^m \phi_i(v_{\sigma(i)}).
\end{equation}

Likewise for $G = \O(V)$ or $\Sp(V)$, there is a canonical isomorphism of $G$-modules given by
\begin{align}
\label{iso O Sp}
\begin{split}
    V^{\otimes m} &\longrightarrow \C[V^m]_{\mathbf{1}},\\
    w_1 \otimes \cdots \otimes w_m &\longmapsto \left[ (v_1, \ldots, v_m) \mapsto \begin{cases} 
    \prod_{i=1}^m b(v_i, w_i), & G = \O(V),\\[1ex]
    \prod_{i=1}^m \omega(v_i, w_i), & G = \Sp(V)
    \end{cases}
    \right]
    \end{split}
\end{align}
and extended by linearity.
By parts~\ref{O in poly thm} and~\ref{Sp in poly thm} of Theorem~\ref{thm:poly invariants GL O Sp}, we see that $\C[V^m]^G_{\mathbf{1}}$ is nonzero if and only if $m$ is even, since this is the only way an arc diagram can be 1-regular.
Therefore from now on we assume that $m$ is even.
The set of 1-regular arc diagrams is nothing other than the set of perfect matchings on $[m]$.
By Theorem~\ref{thm:poly invariants GL O Sp}, the perfect matching $\{\{i_1, j_1\}, \ldots, \{i_{m/2}, j_{m/2}\}\}$, where each $i_k < j_k$, represents the invariant polynomial $\prod_{k=1}^{m/2} f_{i_k,j_k}$, whose preimage under the isomorphism~\eqref{iso O Sp} is
\begin{equation}
    \label{preimage O Sp}
    \sum_{\mathbf{c} \in [n]^{m/2}}
    \cdots \otimes \posarrow{e^{\phantom{c_k}}_{c_k}}{i_k} \otimes \cdots \otimes \posarrow{e^{c_k}_{\phantom{c_k}}}{j_k} \otimes \cdots,
\end{equation}
where $\{e_1, \ldots, e_n\}$ and $\{e^1, \ldots, e^n\}$ are dual bases of $V$ with respect to the form $b$ or $\omega$, and where the $i_k$ and $j_k$ designate tensor factors.
Just as for $\GL(V)$ above, there is also a coordinate-free interpretation whereby the perfect matching which represents the tensor~\eqref{preimage O Sp} is instead viewed as the functional on $V^{\otimes m}$ defined as follows (see details in~\cite{GW}*{\S5.3.2}):
\begin{equation}
    \label{preimage O Sp dual}
    v_1 \otimes \cdots \otimes v_m \longmapsto \begin{cases} 
    \prod_{k=1}^{m/2} b(v_{i_k}, v_{j_k}), & G = \O(V), \\[1ex]
    \prod_{k=1}^{m/2} \omega(v_{i_k}, v_{j_k}), & G = \Sp(V).
    \end{cases}
\end{equation}

The 1-regular condition obviates the need to distinguish between strict and weak nestings, since no two arcs can share a vertex.
Hence in the following two theorems, we drop the qualifiers ``strict'' and ``weak.''

\begin{theorem}[Tensor invariants for $\GL(V)$, $\O(V)$, and $\Sp(V)$]\
\label{thm:multilinear GL O Sp}

\begin{enumerate}[label=\textup{(\alph*)}]
    \item \label{GL in tensor thm} Let $\dim V = n$.
    A linear basis for $(V^{\otimes m} \otimes V^{* \otimes m})^{\GL(V)}$ is given by the set of $(m,m)$-bipartite, 1-regular arc diagrams that do not contain an $(n+1)$-nesting.
    The arc diagram given by the permutation $\sigma \in \mathfrak{S}_m$ represents the tensor~\eqref{preimage GL}, or equivalently, the functional~\eqref{preimage GL dual}.

    \item \label{O in tensor thm} Let $\dim V = n$.
    A linear basis for $(V^{\otimes m})^{\O(V)}$ is given by the set of 1-regular arc diagrams on $[m]$ in which every collection of $n+1$ many arcs contains a nested pair.
    The matching $\{ \{i_1, j_1\}, \ldots, \{i_{m/2}, j_{m/2}\}\}$ represents the tensor~\eqref{preimage O Sp}, or equivalently, the functional~\eqref{preimage O Sp dual}.

    \item \label{Sp in tensor thm} Let $\dim V = 2n$.
    A linear basis for $(V^{\otimes m})^{\Sp(V)}$ is given by the set of 1-regular arc diagrams on $[m]$ that do not contain an $(n+1)$-nesting.
    The matching $\{ \{i_1, j_1\}, \ldots, \{i_{m/2}, j_{m/2}\}\}$ represents the tensor~\eqref{preimage O Sp}, or equivalently, the functional~\eqref{preimage O Sp dual}.

    \end{enumerate}
\end{theorem}

To apply the program above to the subgroups $\SL(V)$ and $\SO(V)$, we simply account for the presence of hyperedges.
Taking the dual perspective that gave us the coordinate-free interpretations~\eqref{preimage GL dual} and~\eqref{preimage O Sp dual} above, we again view each diagram as the product of its arcs and hyperedges, where for $\SL(V)$ we have
\begin{align}
    \label{preimage SL}
    \begin{split}
    \text{arc $\{i^*,j\}$} & \leadsto [v_1 \otimes \cdots v_p \otimes \phi_1 \otimes \cdots \otimes \phi_q \mapsto \phi_j(v_i)],\\
    \text{hyperedge $\{a^*_1, \ldots, a^*_n\}$} & \leadsto [v_1 \otimes \cdots v_p \otimes \phi_1 \otimes \cdots \otimes \phi_q \mapsto \det(v_{a_1}, \ldots, v_{a_n})],\\
    \text{hyperedge $\{b_1, \ldots, b_n\}$} & \leadsto [v_1 \otimes \cdots v_p \otimes \phi_1 \otimes \cdots \otimes \phi_q \mapsto \det(\phi_{b_1}, \ldots, \phi_{b_n})],
    \end{split}
\end{align}
and for $\SO(V)$ we have
\begin{align}
    \label{preimage SO}
    \begin{split}
    \text{arc $\{i,j\}$} & \leadsto [v_1 \otimes \cdots \otimes  v_m \mapsto b(v_i, v_j)],\\
   \text{hyperedge $\{a_1, \ldots, a_n\}$} & \leadsto [v_1 \otimes \cdots \otimes v_m \mapsto \det(v_{a_1}, \ldots, v_{a_n})].
    \end{split}
\end{align}

\begin{theorem}[Tensor invariants for $\SL(V)$ and $\SO(V)$]
\label{thm:multilinear SL SO}

Let $\dim V = n$.

\begin{enumerate}[label=\textup{(\alph*)}]
    \item \label{SL in tensor thm}
    A linear basis for $(V^{\otimes p} \otimes V^{*\otimes q})^{\SL(V)}$ is given by the set of arc diagrams described in Theorem~\ref{thm:poly invariants SL SO}\ref{GL in tensor thm} that are 1-regular.
    Each such diagram represents the product of its arcs and hyperedges described in~\eqref{preimage SL}.

    \item \label{SO in tensor thm} 
    A linear basis for $(V^{\otimes m})^{\SO(V)}$ is given by the set of arc diagrams described in Theorem~\ref{thm:poly invariants SL SO}\ref{O in tensor thm} that are 1-regular.
    Each such diagram represents the product of its arcs and (possibly) hyperedge, as described in~\eqref{preimage SO}.
\end{enumerate}
    
\end{theorem}

\section{Equivalent enumerations of basis elements}
\label{sec:enumeration}

The 1-regular arc diagrams and hyperdiagrams we described in the previous section seem to give the most natural model for linear bases of tensor invariants, since arcs between pairs of vertices are viewed as contractions between pairs of vectors (or a vector and covector) in the obvious way.
If, however, one is merely interested in \emph{enumerating} basis elements --- that is, computing the dimensions of the $G$-invariant subspaces --- then it is natural to seek other, perhaps more familiar combinatorial objects which are in bijection with these bases.
As it turns out, these dimensions can be counted in terms of several quite different types of objects.
Foremost among these are standard Young tableaux, which afford a uniform description of the dimension of the tensor invariants for all of the classical groups.
As far as possible, we also include in our theorems other combinatorial interpretations in terms of well-known objects.

Recall that a \emph{standard Young tableau} of shape $\lambda \vdash m$ is a semistandard tableau whose entries are $1, \ldots, m$, each occurring exactly once.
Equivalently, a semistandard tableau is standard if and only if it has weight $\mathbf{1}$.
Let $\SYT(\lambda)$ denote the set of standard Young tableaux of shape $\lambda$.
It is well known~\cite{Fulton}*{p.~53} that the number of these tableaux is given by the hook length formula $\# \SYT(\lambda) = m! / \prod_{(i,j) \in \lambda}{h_\lambda(i,j)}$, where $h_\lambda(i,j)$ denotes the length of the hook determined by the box $(i,j)$ in the Young diagram of $\lambda$.

\begin{theorem}[Dimension descriptions for $\GL(V)$, $\O(V)$, and $\Sp(V)$]\
\label{thm:enumerate GL O Sp}

\begin{enumerate}[label=\textup{(\alph*)}]
    \item \label{GL in dim thm} Let $\dim V = n$.
    The dimension of $(V^{\otimes m} \otimes V^{*\otimes m})^{\GL(V)}$ equals 

    \begin{enumerate}[label=\textup{(\arabic*)}]
    
    \item the number of arc diagrams described in Theorem~\ref{thm:multilinear GL O Sp}\ref{GL in tensor thm};
    
    \item $ \displaystyle
    \displaystyle \sum_{\substack{\lambda \vdash m, \\ \ell(\lambda) \leq n}} \#\SYT(\lambda)^2$;

    \item the number of permutations of $[m]$ whose decreasing subsequences have length at most~$n$;

    \item the number of walks in the complement of the upper order ideal generated by $(1^{n+1})$ in Young's lattice, such that the walk begins at the empty diagram $\varnothing$, and consists of $m$ steps up followed by $m$ steps down.

    \end{enumerate}

    \item \label{O in dim thm} Let $\dim V = n$.
    The dimension of $(V^{\otimes m})^{\O(V)}$ equals
    
    \begin{enumerate}[label=\textup{(\arabic*)}]
    
    \item the number of arc diagrams described in Theorem~\ref{thm:multilinear GL O Sp}\ref{O in tensor thm};
    
    \item \label{SYT for O} $\displaystyle \sum_{\mathclap{\substack{\lambda \vdash m, \\ \ell(\lambda) \leq n, \\ \textup{even rows}}}} \#\SYT(\lambda)$;
    
    \item the number of fixed-point-free involutions on $[m]$ whose increasing subsequences have length at most $n$.
\end{enumerate}

    \item \label{Sp in dim thm} Let $\dim V = 2n$.
    The dimension of $(V^{\otimes m})^{\Sp(V)}$ equals

    \begin{enumerate}[label=\textup{(\arabic*)}]

        \item \label{arcs for Sp} the number of arc diagrams described in Theorem~\ref{thm:multilinear GL O Sp}\ref{Sp in tensor thm};
        
        \item \label{SYT for Sp} $\displaystyle \sum_{\mathclap{\substack{\lambda \vdash m, \\ \ell(\lambda) \leq 2n, \\ \textup{even columns}}}} \#\SYT(\lambda)$;

        \item the number of fixed-point-free involutions on $[m]$ whose decreasing subsequences have length at most $2n$;

        \item the number of cm-labeled Dyck paths of semilength $m$ with no singletons and $(n+1)$-noncrossing labels (see definitions in~\cite{Gil20}*{\S4});

        \item the number of oscillating tableaux of length $m$ with at most $n$ columns, which start and end with the empty partition;

        \item the number of walks with $m$ steps, beginning and ending at the origin, inside the intersection of the lattice $\mathbb{Z}^n$ with the cone $x_1 \geq \cdots \geq x_n \geq 0$.

    \end{enumerate}
\end{enumerate}

\end{theorem}

\begin{proof}\

    \begin{enumerate}[label=\textup{(\alph*)}]
        \item \begin{enumerate}[label=\textup{(\arabic*)}]
            \item See the proof of Theorem~\ref{thm:multilinear GL O Sp}\ref{GL in tensor thm}.
            
            \item Specializing Proposition~\ref{prop:standard monomial basis} to multidegree $\mathbf{d} = (\mathbf{1}, \mathbf{1})$, the basis $\mathcal{S}^{\GL(V)}_{(\mathbf{1}, \mathbf{1})}$ consists of bitableaux of weight $(\mathbf{1}, \mathbf{1})$, which are ordered pairs of \emph{standard} Young tableaux.
            Hence in $\mathcal{S}^{\GL(V)}_{(\mathbf{1}, \mathbf{1})}$ we specialize from $\SSYT(\lambda,m) \times \SSYT(\lambda, m)$ to $\SYT(\lambda) \times \SYT(\lambda)$.
            \item This is a well-known property of the $\RSK_{\rm A}$ correspondence restricted to $\mathcal{S}^{\GL(V)}_{(\mathbf{1}, \mathbf{1})}$, that is, to pairs of \emph{standard} Young tableaux; see~\cite{Schensted}*{Thm.~2}.
            \item The upper order ideal generated by $(1^{n+1})$ contains precisely those partitions $\lambda$ such that $\ell(\lambda) > n$.
            For any $\lambda \vdash m$ in the complement, the number of ascending walks from $\varnothing$ to $\lambda$ (which are the walks with $m$ steps) is well known to be $\#\SYT(\lambda)$.
            By symmetry, this is also the number of descending walks from $\lambda$ to $\varnothing$.
            Therefore $\#\SYT(\lambda)^2$ is the number of walks from $\varnothing$ to $\lambda$ to $\varnothing$, consisting of $m$ steps up followed by $m$ steps down.
        \end{enumerate}

        \item \label{proof O tensor b} \begin{enumerate}[label=\textup{(\arabic*)}]

            \item See the proof of Theorem~\ref{thm:multilinear GL O Sp}\ref{O in tensor thm}.
            
            \item Again, specializing Proposition~\ref{prop:standard monomial basis} to multidegree $\mathbf{1}$ is the same as restricting to $\SYT$s.
            \item \label{proof O tensor b2} Let $T \in \SYT(\lambda)$, where $\ell(\lambda) \leq n$ and $\lambda \vdash m$ has even rows.
            Let $T'$ denote the SYT obtained by transposing $T$.
            Then $T'$ has exactly $\ell(\lambda)$ many columns, all of which have even length.
            Let $M \coloneqq \RSK_{\rm C}(T')$, which is an $m \times m$ symmetric permutation matrix with trace zero (i.e., a fixed-point-free involution on $[m]$).
            By the corollary on page 718 of Knuth~\cite{Knuth}, combined with his description of $\RSK_{\rm C}$ in his Theorem 4, the number of columns in $T'$ (namely $\ell(\lambda) \leq n$) is the length of the longest increasing subsequence in this involution.
            This procedure is invertible since $\RSK_{\rm C}$ is a bijection.
        \end{enumerate}

        \item \begin{enumerate}[label=\textup{(\arabic*)}]

            \item See the proof of Theorem~\ref{thm:multilinear GL O Sp}\ref{Sp in tensor thm}.
        
            \item Again, specializing Proposition~\ref{prop:standard monomial basis} to multidegree $\mathbf{1}$ is the same as restricting to $\SYT$s.

            \item Let $T \in \SYT(\lambda)$, where $\ell(\lambda) \leq 2n$ and $\lambda \vdash m$ has even columns.
            By Proposition~\ref{prop:RSK} and by~\eqref{RSK width height}, the matrix $M \coloneqq \RSK_{\rm C}(T)$ is an $m \times m$ symmetric permutation matrix with trace zero, such that ${\rm supp}(M) \cap P$ has width at most $n$.
            Therefore each such $T$ corresponds to a fixed-point-free involution on $[m]$ given by $M$, where $M$ contains at most $n$ many 1's lying along a single path from southwest to northeast \emph{above} the diagonal.
            Since $M$ is symmetric, it contains at most $2n$ many 1's lying along a single path from southeast to northwest.
            This is equivalent to the longest decreasing subsequence in the corresponding involution.
            This procedure is invertible since $\RSK_{\rm C}$ is a bijection.

            \item This is equivalent to~\ref{SYT for Sp} by~\cite{Gil20}*{Prop.~4.2}.

            \item This is equivalent to~\ref{SYT for Sp} by~\cite{Krattenthaler}*{Thm.~4}. 
            
            \item This is a special case of~\cite{GrabinerMagyar}*{Thm.~2 (p.~244) and \S6.2}, which implies that $\dim (V^{\otimes m})^{\Sp(V)}$ equals the number of $m$-step walks beginning and ending at $\rho = (n, n-1, \ldots, 2,1)$, confined to the interior $\{ \mathbf{x} \in \mathbb{Z}^n : x_1 > \cdots > x_n > 0\}$ of the dominant Weyl chamber of type $C_n$ (i.e., for $\mathfrak{g} = \mathfrak{sp}_{2n}$).
            Translating by $-\rho$ changes the start and endpoint to the origin, and changes the strict inequalities on $\mathbf{x}$ to weak inequalities.
            (There is also a generating function given in~\cite{GrabinerMagyar}*{eqn.~(39)}.) \qedhere
        \end{enumerate}
    \end{enumerate}
\end{proof}

\begin{rem}
    In the special case $G = \O(3)$, Theorem~\ref{thm:enumerate GL O Sp}\ref{O in dim thm}\ref{SYT for O} can be found in the physics paper~\cite{Smith1985}*{p.~654} on isotropic tensors; see also the earlier work~\cite{Smith1968}.
\end{rem}

\begin{rem}
\label{rem:stable}
    When $\dim V$ is sufficiently large, the length condition $\ell(\lambda) \leq \dim V$ becomes vacuous, and the enumerations in Theorem~\ref{thm:enumerate GL O Sp} no longer depend on $n$.
    We call this the \emph{stable range}:
    \[
    \begin{cases}
        n \geq m, & G = \GL(V),\\
        n \geq m/2, & G = \O(V) \text{ or } \Sp(V).
    \end{cases}
    \]
    To verify the second case above, we observe that for $\O(V)$, any partition $\lambda \vdash m$ with even row lengths must have length at most $\frac{m}{2}$; on the other hand, for $\Sp(V)$, the even-column condition still allows partitions of length at most $m$, so that the stable range occurs when $2n \geq m$.
    In the stable range, the formulas in Theorem~\ref{thm:enumerate GL O Sp} simplify to the following:
    \begin{align*}
        \dim (V^{\otimes m} \otimes V^{*\otimes m})^{\GL(V)}=\dim {\rm End}_{\GL(V)} (V^{\otimes m}) &= m!,\\
        \dim (V^{\otimes m})^{\O(V)} = \dim (V^{\otimes m})^{\Sp(V)} &= \begin{cases} (m-1)!!, & \text{$m$ even},\\
        0, & \text{$m$ odd}.
        \end{cases}
    \end{align*}
    (See the bottom row of Table~\ref{table:OEIS}.)
    These dimensions are interpreted in~\cite{CEW}*{Cor.~7, Ex.~13} in terms of contingency tables.
    See also the unpublished survey~\cite{Callan} for a wealth of other combinatorial interpretations of the odd double factorial.
    \end{rem}

\begin{theorem}[Dimension descriptions for $\SL(V)$ and $\SO(V)$]
\label{thm:enumerate SL SO}

Let $\dim V = n$. 

\begin{enumerate}[label=\textup{(\alph*)}]
    \item \label{SL in dim thm} Let $t \coloneqq \frac{|p-q|}{n}$.
    If $t \notin \mathbb{N}$, then $(V^{\otimes p} \otimes V^{*\otimes q})^{\SL(V)}$ has dimension zero.
    If $t \in \mathbb{N}$, then the dimension of $(V^{\otimes p} \otimes V^{*\otimes q})^{\SL(V)}$ equals    \begin{enumerate}[label=\textup{(\arabic*)}]
        \item \label{arcs for SL} the number of arc diagrams described in Theorem~\ref{thm:multilinear SL SO}\ref{SL in tensor thm};

        \item \label{SYT for SL} $\displaystyle \sum_{\mathclap{\substack{\lambda \vdash \min\{p,q\}, \\ \ell(\lambda) \leq n}}} \#\SYT(\lambda + t^n) \cdot \#\SYT(\lambda)$,
        
        where $\lambda + t^n$ is the shape obtained from $\lambda$ by adding $t$ boxes to all $n$ rows.
    \end{enumerate}

    \item \label{SO in dim thm} The dimension of $(V^{\otimes m})^{\SO(V)}$ equals zero if and only if $m$ is odd and $n$ is even or greater than $m$.
    Otherwise, the dimension equals

    \begin{enumerate}[label=\textup{(\arabic*)}]
    \item \label{arcs for SO} the number of arc diagrams described in Theorem~\ref{thm:multilinear SL SO}\ref{SO in tensor thm};
    
    \item \label{proof SO tensor 1} $\displaystyle \sum_{\mathclap{\substack{\lambda \vdash m, \\ \ell(\lambda) \leq n, \\ \textup{even rows}}}} \#\SYT(\lambda) 
    + \sum_{\mathclap{\substack{\lambda \vdash m-n, \\ \ell(\lambda) \leq n, \\ \textup{even rows}}}} \#\SYT(\lambda + 1^n)$;
    
    \item the number of involutions on $[m]$ whose increasing subsequences have length at most $n$, and which have either $0$ or $n$ fixed points;

    \item \textup{(if $n = 2r$ is even)} the number of walks with $m$ steps, beginning and ending at the origin, inside the intersection of the lattice $\mathbb{Z}^r$ with the cone $x_1 \geq \cdots \geq x_{r-1} \geq |x_r|$.

    \end{enumerate}
    \end{enumerate}

\end{theorem}

\begin{proof}\

\begin{enumerate}[label=\textup{(\alph*)}]
    \item 
    \begin{enumerate}[label=\textup{(\arabic*)}]
        \item See the proof of Theorem~\ref{thm:multilinear SL SO}\ref{SL in tensor thm}.

        \item Specializing to multidegree $\mathbf{d} = (\mathbf{1}, \mathbf{1})$, the basis $\mathcal{S}^{\SL(V)}_{(\mathbf{1}, \mathbf{1})}$ is obtained by changing ``SSYT'' to ``SYT'' in Proposition~\ref{prop:standard monomials SL SO}\ref{subprop:standard monomials SL SO a}.
        In that proposition, if $p \geq q$, then the parameter $r \geq 0$ is an integer such that $p = q+rn$; if $p < q$, then the parameter $s>0$ is an integer such that $p + sn = q$.
        Let $u$ denote this parameter $r$ or $s$, as the case may be; then $\mathcal{S}^{\SL(V)}_{(\mathbf{1}, \mathbf{1})}$ is nonempty if and only if $u \in \mathbb{N}$ is such that $\min\{p,q\} + un = \max\{p,q\}$, or equivalently, $u = \frac{|p-q|}{n}$.
        The result now follows by substituting $t$ for $u$.
    
    \end{enumerate}

    \item \begin{enumerate}[label=\textup{(\arabic*)}]

        \item See the proof of Theorem~\ref{thm:multilinear SL SO}\ref{SO in tensor thm}. 
        
        \item Specializing to multidegree $\mathbf{d} = \mathbf{1}$, the basis $\mathcal{S}^{\SO(V)}_{\mathbf{1}}$ is obtained by changing ``SSYT'' to ``SYT'' in Proposition~\ref{prop:standard monomials SL SO}\ref{subprop:standard monomials SL SO b}.
        The set of even-rowed shapes $\lambda \vdash m$ is nonempty if and only if $m$ is even;
        the set of even-rowed shapes $\lambda \vdash m-n$ is nonempty if and only if $m-n$ is nonnegative and even.
        
        \item By Theorem~\ref{thm:enumerate GL O Sp}\ref{proof O tensor b}\ref{proof O tensor b2}, the first sum in part~\ref{proof SO tensor 1} equals the number of involutions on $[m]$ whose increasing subsequences have length at most $n$ and which have no fixed points.
        Now we consider the second sum,
        where the shapes in question have the form $\lambda + 1^n$, namely, all $n$ rows have odd length.
        Consider the transpose $T'$ of each tableau $T$ appearing in this sum; each $T'$ has $n$ columns, all of odd length.
        By Theorem 4 in Knuth~\cite{Knuth}, there is a generalization of $\RSK_{\rm C}$ giving a bijection $T' \mapsto M$ between SSYTs and involutions on $[m]$ with increasing subsequences no longer than $n$, such that the number of odd-length columns in $T$ equals the number of fixed points of $M$.
        The result follows by combining the first and second sums.
        
        \item This is a special case of~\cite{GrabinerMagyar}*{Thm.~2 (p.~244) and \S6.3}, which implies that $\dim(V^{\otimes m})^{\SO(V)}$ equals the number of $m$-step walks beginning and ending at $\rho = (r-1, r-2, \ldots, 1, 0)$, confined to the interior $\{\mathbf{x} \in \mathbb{Z}^r : x_1 > \cdots > x_{r-1} > |x_r|\}$ of the dominant Weyl chamber of type $D_r$ (i.e., for $\mathfrak{g} = \mathfrak{so}_{2r}$).
        Translating by $-\rho$ changes the start and endpoint to the origin, and changes the strict inequalities on $\mathbf{x}$ to weak inequalities.
        (There is also a generating function given in~\cite{GrabinerMagyar}*{eqn.~(43)}.) \qedhere  
    \end{enumerate}
       
\end{enumerate}
\end{proof}

\begin{rem}
\label{rem:special stable}
    In parts~\ref{SL in dim thm} and~\ref{SO in dim thm} of Theorem~\ref{thm:enumerate SL SO}, if $n > \max\{p,q\}$ or $n > m$ (respectively), then the dimension of the $\SL(V)$- or $\SO(V)$-invariants is the same as for $\GL(V)$ or $\O(V)$ in parts~\ref{GL in dim thm} and~\ref{O in dim thm} of Theorem~\ref{thm:enumerate GL O Sp}.
    This is the range in which it is impossible for the arc diagram to contain an order-$n$ hyperedge.
    For $\SL(V)$, in the special case where $p=0$ or $q=0$, Theorem~\ref{thm:enumerate SL SO}\ref{SL in dim thm}\ref{SYT for SL} implies that 
    \[
    \dim (V^{\otimes m})^{\SL(V)}=\dim (V^{*\otimes m})^{\SL(V)} = \begin{cases} 
    \#\SYT(t^n), & t \coloneqq \frac{m}{n} \in \mathbb{N},\\
    0 & \text{otherwise.}
    \end{cases}
    \]
    These numbers $\#\SYT(t^n)$ are called the $n$-dimensional Catalan numbers; see also~\cites{Gil20, BostanEtAl}, as well as the random walk interpretation in~\cite{GrabinerMagyar}*{top of p.~256}.
\end{rem}

\section*{Appendix: Tables}

In Table~\ref{table:OEIS}, we collect known sequences obtained by listing off the dimensions 
in Theorems~\ref{thm:enumerate GL O Sp} and~\ref{thm:enumerate SL SO} for fixed values of $n$.
For each sequence we give a brief description and its identifier in the Online Encyclopedia of Integer Sequences (OEIS).
Upon fixing one of the five classical groups (with the exception of $\SL(V)$, due to the double parameters $p$ and $q$), one obtains a ``sequence of sequences'' from the OEIS, which approaches a limiting sequence as $n$ approaches~$\infty$ (see the bottom row, and compare with Remarks~\ref{rem:stable} and~\ref{rem:special stable}).

In Tables~\ref{table:SL example}--\ref{table:Sp example}, we present concrete examples illustrating some of the sequences in Table~\ref{table:OEIS} in terms of tensor invariants, arc diagrams, and standard Young tableaux.
Recalling Proposition~\ref{prop:standard monomial basis} and Definition~\ref{def: M SL SO}, we emphasize that the RSK correspondence yields an explicit bijection $\mathcal{B}^{G}_{\mathbf{d}} \longleftrightarrow \mathcal{S}^G_{\mathbf{d}}$, that is, between the arc diagrams in Theorems~\ref{thm:multilinear GL O Sp} and~\ref{thm:multilinear SL SO}, on one hand, and the standard Young tableaux in Theorems~\ref{thm:enumerate GL O Sp} and~\ref{thm:enumerate SL SO} on the other hand.
Therefore, in Tables~\ref{table:SL example}--\ref{table:Sp example}, we depict each tableau directly below the arc diagram to which it corresponds via RSK.

\newpage

\begin{sidewaystable}[ph!]
    \centering
    \resizebox{\linewidth}{!}{
    \begin{tblr}{colspec={|Q[m,c]|Q[m,c]|Q[m,c]|Q[m,c]|Q[m,c]|Q[m,c]|},stretch=1.5}

\hline

$n$ & $\GL(V)$ & $\SL(V)$ & {$\O(V)$ \\ (interlace with 0's)} & $\SO(V)$ & {$\Sp(V)$ \\ (interlace with 0's)} \\

\hline[2pt]

$1$ & {\href{https://oeis.org/A000012}{A000012} \\ $(1,1,1,\ldots)$} & {\href{https://oeis.org/A000012}{A000012} \\ $(1,1,1,\ldots)$} & {\href{https://oeis.org/A000012}{A000012} \\ $(1,1,1,\ldots)$} & {\href{https://oeis.org/A000012}{A000012} \\ $(1,1,1,\ldots)$} & {\href{https://oeis.org/A000108}{A000108} \\ $(1, 2, 5, 14, 42, 132, \ldots)$ \\ Catalan numbers} \\

\hline

$2$ & {\href{https://oeis.org/A000108}{A000108} \\ $(1, 2, 5, 14, 42,132, \ldots)$ \\ Catalan numbers} & {\href{https://oeis.org/A126120}{A126120} \\ $(0,1,0,2,0,5,\ldots)$ \\ Catalan numbers with 0's \\ (index = $p+q$)} & {\href{https://oeis.org/A001700}{A001700} \\ \href{https://oeis.org/A088218}{A088218} \\ $(1, 3, 10, 35, 126, 462, \ldots)$} & {\href{https://oeis.org/A126869}{A126869} \\ (or \href{https://oeis.org/A000984}{A000984} with 0's) \\ $(0, 2, 0, 6, 0, 20, \ldots)$ \\ Central binomial coefficients} & {\href{https://oeis.org/A005700}{A005700} \\ $(1, 3, 14, 84, 594, 4719, \ldots)$ } \\

\hline

$3$ & {\href{https://oeis.org/A005802}{A005802} \\ $(1, 2, 6, 23, 103, 513, \ldots)$} & Depends on $p$ and $q$ & {\href{https://oeis.org/A099251}{A099251} \\ $(1, 3, 15, 91, 603, 4213, \ldots)$ \\ Bisection of Motzkin sums} & {\href{https://oeis.org/A005043}{A005043} \\ $(0, 1, 1, 3, 6, 15, \ldots)$ \\ Riordan numbers} & 
{\href{https://oeis.org/A136092}{A136092} \\(or \href{https://oeis.org/A138540}{A138540} without 0's) \\ $(1, 3, 15, 104, 909, 9449, \ldots)$} \\

\hline

$4$ & {\href{https://oeis.org/A047889}{A047889} \\ $(1, 2, 6, 24, 119, 694, \ldots)$} & Depends on $p$ and $q$ & {\href{https://oeis.org/A246860}{A246860} \\ $(1, 3, 15, 105, 903, 8778, \ldots)$} & {\href{https://oeis.org/A001246}{A001246} with 0's \\ $(0,1, 0,4, 0,25, \ldots)$ \\ Squared Catalan numbers} & {\href{https://oeis.org/A251598}{A251598}* \\ $(1, 3, 15, 105, 944, 10340, \ldots)$} \\

\hline

$5$ & {\href{https://oeis.org/A047890}{A047890} \\ $(1, 2, 6, 24, 120, 719, \ldots)$} & Depends on $p$ and $q$ & {\href{https://oeis.org/A247304}{A247304}* \\ $(1, 3, 15, 105, 945, 10263, \ldots)$} & {\href{https://oeis.org/A095922}{A095922}* \\ $(0, 1, 0, 3, 1, 15, \ldots)$} & { \\ $(1, 3, 15, 105, 945, 10394, \ldots)$} \\

\hline[2pt]

$\infty$ & {\href{https://oeis.org/A000142}{A000142} \\ $(1, 2, 6, 24, 120, 720, \ldots)$ \\ Factorials} & $\begin{cases}
   \href{https://oeis.org/A000142}{A000142},& p=q,\\ (0,0,0,\ldots), & p \neq q \end{cases}$ & \SetCell[c=3]{c}{{\href{https://oeis.org/A123023}{A123023} (or \href{https://oeis.org/A001147}{A001147} with 0's) \\ $(0,1, 0,3, 0,15, 0,105, 0,945, 0,10395, \ldots)$  \\ Odd double factorials}
   } 
   \\

   \hline

\end{tblr}
}
    \caption{OEIS entries enumerating the dimensions in Theorems~\ref{thm:enumerate GL O Sp} and~\ref{thm:enumerate SL SO}.
    For all groups except $\SL(V)$ the index in the sequences is $m$.
    For $\O(V)$ and $\Sp(V)$, the phrase ``interlace with 0's'' means that the sequence $(a_1, a_2, \ldots)$ should  be read as $(0,a_1, 0, a_2, \ldots)$.
    The bottom row (see Remarks~\ref{rem:stable} and~\ref{rem:special stable}) gives the limit as $n \rightarrow \infty$.
    The asterisks indicate the entries for which no other combinatorial interpretation is given in the OEIS besides the dimension of the invariant subspace as computed in LiE; for most of these small values of $n$, however, the OEIS entries list many other combinatorial interpretations.}
    \label{table:OEIS}
\end{sidewaystable}

\ytableausetup{boxsize=.9em,nocentertableaux}

\newpage

\begin{table}[h!]   
\centering
    \input{SL_big_example}
    \caption{For $G = \SL(V)$ with $n=\dim V=4$, we use parts~\ref{arcs for SL} and~\ref{SYT for SL} of Theorem~\ref{thm:enumerate SL SO}\ref{SL in dim thm} to compute $\dim (V^{\otimes p} \otimes V^{*\otimes q})^{\SL(V)}$ for all $p,q$ such that $p+q = 8$.
    This count is automatically zero unless $t \coloneqq \frac{|p-q|}{4}$ is an integer; this leaves us with the three possibilities given in the three rows of this table (corresponding to $t = 0$, $1$, or $2$).
    In the diagrams, we assume (without loss of generality) that $p \geq q$.
    The last column in the table enumerates the arc diagrams (equivalently, tableau pairs) to give the dimension of $(V^{\otimes p} \otimes V^{*\otimes q})^{\SL(V)}$.
    Note that the arc diagrams (equivalently, the tableau pairs) without any shading --- that is, where $t=0$ --- are precisely those which remain invariants for the full group $\GL(V)$.}
    \label{table:SL example}
\end{table}

\newpage

\begin{table}[h!]
    \centering
    \input{SO_big_example}
    \caption{For $G = \SO(V)$, we fix the tensor order $m=4$ and use parts~\ref{arcs for SO} and~\ref{proof SO tensor 1} of Theorem~\ref{thm:enumerate SL SO}\ref{SO in dim thm} to compute the dimension of $(V^{\otimes 4})^{\SO(V)}$ as $n = \dim V$ varies.
    By inspecting the 4th terms of the sequences listed in the $\SO(V)$ column of Table~\ref{table:OEIS}, we see that the desired sequence is $1,6,3,4,3,3,\ldots$, which stabilizes at $3$ for all $n \geq 5$ (see Remark~\ref{rem:special stable}).
    For each $n$, we depict the arc diagrams (equivalently, the standard Young tableaux) enumerated by this sequence.
    Note that a hyperedge, depicted by shaded vertices, is realized as the shaded first column in the corresponding tableau.
    The arc diagrams (equivalently, the tableaux) without any shading are precisely those which remain invariants for the full group $\O(V)$.}
    \label{table:SO example}
\end{table}

\newpage

\begin{table}[h!]
    \input{Sp_big_example}
    \caption{For $G = \Sp(V)$, we fix the tensor order $m=6$ and use parts~\ref{arcs for Sp} and~\ref{SYT for Sp} of Theorem~\ref{thm:enumerate GL O Sp}\ref{Sp in dim thm} to compute the dimension of $(V^{\otimes 6})^{\Sp(V)}$ as $n$ varies (where $\dim V = 2n$).
    By inspecting the 3rd terms (since $6 \div 2 = 3$, to account for the interlaced zeros) of the sequences listed in the $\Sp(V)$ column of Table~\ref{table:OEIS}, we see that the desired sequence is $5, 14, 15, 15, \ldots$, which stabilizes at $15$ for all $n \geq 3$ (see Remark~\ref{rem:stable}).
    To conserve space, in each row of the table we include only those diagrams which were not counted in the rows above.
    Therefore, each new arc diagram in the $n$th row of the table contains an $n$-nesting, and the length of its corresponding tableau is exactly $2n$.}
    \label{table:Sp example}
\end{table}

\clearpage

\subsection*{Data availability statement}

No datasets were generated or analyzed during the current study.

\subsection*{Declarations}

The authors have no competing interests to declare.
No funds, grants, or other support was received to assist with the preparation of this manuscript.

 \bibliographystyle{amsplain}
 \bibliography{references}

\end{document}